\documentclass[12pt, reqno]{amsart}
\usepackage{amsmath, amsthm, amscd, amsfonts, amssymb, xcolor}
\usepackage{bm}
\usepackage{mathrsfs}

\newtheorem{theorem}{Theorem}[section]
\newtheorem{lemma}[theorem]{Lemma}
\newtheorem{proposition}[theorem]{Proposition}
\newtheorem{corollary}[theorem]{Corollary}

\newtheorem{definition}[theorem]{Definition}
\newtheorem{example}[theorem]{Example}

\numberwithin{equation}{section}

\DeclareMathOperator*{\esssup}{ess\,sup}

\newcommand{\vp}{\varphi}

\newcommand{\clb}{\mathcal{B}}

\newcommand{\cli}{\mathcal{I}}

\newcommand{\clm}{\mathcal{M}}

\newcommand{\cls}{\mathcal{S}}

\newcommand{\D}{\mathbb{D}}

\newcommand{\T}{\mathbb{T}}

\newcommand{\raro}{\rightarrow}

\newcommand{\cam}{\mathscr{M}}

\begin{document}
\setcounter{page}{1}


\title[Analytic primes, $M$-ideals, and $p$-sets]{Analytic primes, $M$-ideals, and $p$-sets in $H^\infty(\mathbb{D})$}

\author[Deepak]{Deepak K. D.}
\address{Department of Applied Mathematics, National Yang Ming Chiao Tung University, Hsinchu City, 30010, Taiwan}
\email{dpk.dkd@gmail.com }

\author[Sarkar]{Jaydeb Sarkar}
\address{Indian Statistical Institute, Statistics and Mathematics Unit, 8th Mile, Mysore Road, Bangalore, 560059,
India}
\email{jay@isibang.ac.in, jaydeb@gmail.com}

\author[Siju]{Sreejith Siju}
\address{Indian Statistical Institute, Statistics and Mathematics Unit, 8th Mile, Mysore Road, Bangalore, 560059,
India}
\email{sreejithsiju5@gmail.com}



\begin{abstract}
We investigate the structure of $p$-sets, $M$-ideals, and a newly introduced notion of analytic primes in $H^\infty(\mathbb{D})$, where $H^\infty(\mathbb{D})$ denotes the Banach algebra of all bounded analytic functions on the open unit disc $\mathbb{D}$ in $\mathbb{C}$. We prove that $M$-ideals in $H^\infty(\mathbb{D})$ are analytic primes and are dense in the Hardy space. Outer functions play a key role in representing closed principal ideals in $H^\infty(\mathbb{D})$ that are $M$-ideals. Some of our results apply to the polydisc. The results presented in this paper offer some new perspectives on $H^\infty(\mathbb{D})$.
\end{abstract}

\maketitle

\tableofcontents

\section{Introduction}\label{sec Intro}

The objective of this work is to integrate several concepts that are rich and of independent interest, including $p$-sets, $M$-ideals, outer functions, function Hilbert spaces, and the Banach algebra $H^\infty(\D)$, to name a few. We also introduce a new notion, namely, analytic primes, and connect it with these classical concepts.

The central topic of this paper is the classical space $H^\infty(\D)$ (see Bers \cite{Bers} and Kakutani \cite{Kakutani}, and the survey by Gamelin \cite{Gamelin1}). Recall that
\[
H^\infty(\D) = \{f \in \text{Hol} (\D): \sup_{z \in \D}|f(z)| < \infty\},
\]
where $\D = \{z \in \mathbb{C}: |z| <1\}$. This is a commutative Banach algebra with respect to the pointwise product and under the supremum norm
\[
\|f\| = \sup_{z \in \D}|f(z)| \qquad (f \in H^\infty(\D)).
\]
This is one of the most important Banach algebras that plays a key role in the theory of Hilbert function spaces and Banach spaces of analytic functions \cite{Hoffman}. And because of its analytic nature, it fits well between commutative $C^*$-algebras and generic commutative Banach algebras. For instance, the Gelfand map $\Gamma : H^\infty(\D) \raro C(\mathcal{M}(H^\infty(\D))$ defined by
\[
\Gamma(f) = \hat{f} \qquad (f \in H^\infty(\D)),
\]
is an isometry, where $\mathcal{M}(H^\infty(\D))$ denotes the maximal ideal space of $H^\infty(\D)$. This is a prototype feature of commutative $C^*$-algebras (see Section \ref{sec Prelim} for more details). However, despite being concrete Banach algebras, this space raises fundamental problems for which there are no definitive solutions. The structure of the closed ideals in $H^\infty(\D)$, for example, is still obscure (however, see \cite{HH, BK}). On one hand, this space has undergone extensive research in the context of Banach spaces of analytic functions (cf. \cite{Burgain, II, KW}). On the other hand, as far as our knowledge extends, $H^\infty(\D)$ has not been investigated in any way from the standpoint of $p$-sets and $M$-ideals. And in this paper, this is exactly what we do.

Hereafter, all the Banach spaces are over the field of complex numbers. Furthermore, projections on a Banach space refer to bounded linear operators $P$ on that space such that
\[
P^2 = P.
\]

\begin{definition}\label{def M ideals}
A closed subspace $C$ of a Banach space $X$ is said to be an $M$-ideal in $X$ if there exists a projection $P: X^* \raro X^*$ such that
\[
\text{ran} P = C^{\perp},
\]
and
\[
\|Px^*\| + \|(I-P)x^*\| = \|x^*\|,
\]
for all $x^* \in X^*$.
\end{definition}

The projection $P$ above is commonly known as an \textit{$L$-projection}. Also, we denote by $C^{\perp}$ the annihilator of $C$:
\[
C^\perp = \{x^* \in X^*: x^*|_C \equiv 0\}.
\]

Alfsen and Effros \cite{AE} introduced the notion of of $M$-ideals as a means of extending the utility of closed two-sided ideals in $C^*$-algebras to Banach spaces. Since then, $M$-ideals have become a useful tool in Banach space theory, and have eventually developed into a subject in their own right (see the monograph \cite{HWW}, and \cite{CJ, GMYZ, Gilles, GL2} and the references therein).

The issue of classifying or representing $M$-ideals in $H^\infty(\D)$ is intricate, and its complexity is comparable to the hierarchy of maximum ideal spaces in $H^\infty(\D)$. Indeed, our first result refers to the ``size" of $M$-ideals in $H^\infty(\D)$ (see Proposition \ref{prop M ideals larger}): Each maximal ideal in $H^\infty(\D)$ that corresponds to a complex homomorphism in $\partial_S H^{\infty}(\D)$ is an $M$-ideal in $H^\infty(\D)$, where $\partial_S H^{\infty}(\D)$ denotes the \v{S}ilov boundary of $\clm(H^\infty(\D))$. In particular, $\{\varphi\}$ is a $p$-set of $H^\infty(\D)$ for each $\varphi \in \partial_S H^{\infty}(\D)$, and hence the maximal ideal corresponding to each complex homomorphism in $\partial_S H^{\infty}(\D)$ is an $M$-ideal in $H^\infty(\D)$. Therefore, it is overly general to expect a detailed description of every $M$-ideal in $H^\infty(\D)$.

We present new insights on $p$-sets (see Definition \ref{def peak set}) and their application to $M$-ideals in $H^\infty(\D)$. For instance (see Theorem \ref{non}): Let $Q$ be a nonempty subset of $\clm(H^\infty(\D))$. If $Q$ is a $p$-set of $H^\infty(\D)$, then
\[
Q \cap \partial_S H^{\infty}(\D) \neq \emptyset.
\]
In particular, $\{\alpha\}$ is not a $p$-set of $H^\infty(\D)$ for all $\alpha \in \D$ (see Corollary \ref{cor M ideal ev}). 

Given a function $f \in H^\infty(\D)$, we denote by $\tilde{f}$ the radial limit extension of $f$ \cite[page 34]{Hoffman}:
\begin{equation}\label{eqn: radial limit}
\tilde{f}(z) = \lim_{r \raro 1^-}f(rz) \qquad (z \in \T \;a.e.).
\end{equation}
The following is an analytic counterpart of the algebraic notion of prime ideals in commutative rings.

\begin{definition}\label{def intro prime}
A proper nontrivial closed subspace $J$ of $H^\infty(\D)$ is called an analytic prime whenever it meets the following condition: Given $f$ and $g$ in $H^\infty(\D)$, if
\[
fg \in J,
\]
and
\[
|\tilde{f}(z)| > \delta \qquad (z \in \T),
\]
for some $\delta > 0$, then
\[
g \in J.
\]
\end{definition}

We have the following (see Theorem \ref{prop1}):

\begin{theorem}\label{thm intro prime}
Proper nontrivial $M$-ideals in $H^\infty(\mathbb{D})$ are analytic primes.
\end{theorem}

We also connect $M$-ideals with inner functions. A function $f \in H^\infty(\D)$ is called \textit{inner} if
\[
|\tilde{f}(z)| =1,
\]
for all $z \in \T$ a.e. \cite[page 62]{Hoffman}. It is clear that the space of inner functions forms a multiplicative set in $H^\infty(\D)$. In Corollary \ref{inn}, we prove that:

\begin{theorem}
A proper and nontrivial $M$-ideal in $H^\infty(\D)$ never intersects the space of inner functions. 
\end{theorem}

We also discuss some relationships between $M$-ideals in $H^\infty(\D)$ and the theory of Hardy space $H^2(\D)$. Recall that
\[
H^2(\D) := \Big\{f \in \text{Hol}(\D): \|f\|:= \Big(\sup_{0<r<1} \int_{\T} |f(rz)|^2 d\mu(z)\Big)^{\frac{1}{2}} < \infty\Big\},
\]
where $d\mu$ denotes the normalized Lebesgue measure on $\T$. In Theorem \ref{thm M is dense}, we prove a Hilbert space property of $M$-ideals in $H^\infty(\D)$. More specifically: If $J$ is a nontrivial $M$-ideal in $H^\infty(\mathbb{D})$, then (note that the inclusion map $i: H^\infty(\D) \hookrightarrow H^2(\D)$ is a contraction)
\begin{equation}\label{thm intro dense}
\overline{J}^{H^2(\D)} = H^2(\mathbb{D}).
\end{equation}
This norm density property is a notable characteristic of $M$-ideals in $H^\infty(\D)$. This is indeed in line with the norm density of the ring of polynomials:
\[
\overline{\mathbb{C}[z]}^{H^2(\D)} = H^2(\mathbb{D}).
\]

Now we turn to the principal ideals or singly generated ideals in $H^\infty(\D)$. Given $f \in H^\infty(\D)$, we denote by $I_{H^\infty(\D)}(f)$ the closed ideal of $H^\infty(\D)$ generated by $f$. That is,
\begin{equation}\label{eq prin ideal}
I_{H^\infty(\D)}(f) = \overline{\{fg : g \in H^\infty(\D)\}}^{H^\infty(\D)}.
\end{equation}
We remind the reader that the classification of principal ideals in $H^\infty(\D)$ is unknown. Our goal here is to study the $M$-ideal structure of $I_{H^\infty(\D)}(f)$. We present an operator theoretic necessary condition (see Corollary \ref{cor ker of TO}): If $I_{H^\infty(\D^n)}(\vp)$ is an $M$-ideal in $H^\infty(\mathbb{D}^n)$, then
\[
\ker T_\vp^* = \{0\}.
\]
Recall that $T_\vp$ on $H^2(\D)$ is the \textit{Toeplitz operator} with symbol $\vp$:
\[
T_\vp f = \vp f \qquad (f \in H^2(\D)).
\]
It is worth noting that $\ker T_{\vp}^* = \ker T_{\bar{\vp}}$. Kernels of Toeplitz operators are recognised as complex objects and a subject of independent interest \cite{DS1}. Therefore, $M$-ideals in $H^\infty(\D)$ touch on the delicacy of the kernels of Toeplitz operators.

We continue with more Hilbert function space techniques and results. Let $f \in H^\infty(\D)$ (or even $f \in H^2(\D)$) be a nonzero function. Then there exist an inner function $f_I$ and an outer function $f_O$ (unique up to a constant of modulus one) such that \cite[page 67]{Hoffman}
\begin{equation}\label{eqn IO fac}
f = f_I f_O.
\end{equation}
Recall that a function $f \in H^2(\D)$ is \textit{outer} if
\[
\overline{\text{span}\{z^m f: m \geq 0\}}^{H^2(\D)} = H^2(\D).
\]

We have already talked about the role of inner functions in $M$-ideals. The following connects $M$-ideals with outer functions (see Corollary \ref{c2} for more details):

\begin{theorem}\label{converse}
Let $f \in H^\infty(\D)$. If $I_{H^\infty(\D)}(f)$ is an $M$-ideal in $H^\infty(\D)$, then $f$ is an outer function.
\end{theorem}

All the results from Theorem \ref{thm intro prime} to the theorem stated above are applicable in the polydisc setting and are consolidated in Section \ref{sec prime}.

The above result also raises the question of the extent to which the converse is true. Apparently, the answer to this problem is difficult. However, with the help of a result by Izuchi and under an additional condition, namely, Jensen's equality (see Section \ref{sec Principal ideal}), Theorem \ref{thm: conv Jensen} provides the following characterization:

\begin{theorem}
Let $f\in H^\infty(\D)$. Then $I_{H^\infty(\D)}(f)$ is an $M$-ideal in $H^\infty(\D)$ if and only if $f$ is an outer function and satisfies Jensen's equality.
\end{theorem}

Note that Jensen's equality is itself defined in terms of the maximal ideal space, and no purely function-theoretic characterization of it is currently known. Since the other known characterizations of $M$-ideals also rely on the maximal ideal space, a natural question is whether one can find a purely function-theoretic characterization of principal $M$-ideals. This problem appears to be difficult. Nevertheless, we introduce a large class of bounded analytic functions defined solely in terms of their boundary behavior, for which the principal ideals they generate are $M$-ideals. To describe this class, we introduce the notion of essential zeros of $H^\infty(\D)$-functions, which may be of independent interest.

\begin{definition}\label{def ZTf}
We say that $z\in \T$ is an essential zero of an outer function $f\in H^\infty(\D)$ if, for any $\varepsilon>0$ and every open neighborhood $N_z\subset \T$ of $z$, we have the following:
\[
\mu(|\tilde f|^{-1}(0, \varepsilon)\cap N_z) \neq 0.
\]
We let $Z_\T(f)$ denote the essential zeros of the outer function $f$.
\end{definition}

Clearly, $Z_\T(f)$ is closed. Moreover, if $Z_\T(f)=\emptyset$, then $f$ is invertible. The example presented in \eqref{eqn def of f for zero} serves to illustrate the computation of zero sets of outer functions. The notion of zero sets of functions in $ H^\infty(\D)$ relies on the outer part as follows (see \eqref{eqn IO fac}): The \textit{zero set} of a function $f\in H^\infty(\D)$ is defined by
\[
Z_\T(f) = \{z\in \T: z\in Z_\T(f_O)\}.
\]

\begin{definition}\label{def ZD}
Define $Z^\infty(\D)$ as the set of all functions $f$ in $H^\infty(\D)$ that have a continuous extension to $\D\cup Z_\T(f)$.
\end{definition}

It is also evident that $A(\D) \subseteq Z^\infty(\D)$, where $A(\D)$ is the disc algebra defined by
\begin{equation}\label{eq disc alg}
A(\D) = \{f \in C(\overline{\D}): f|_{\D} \in \text{Hol}(\D)\}.
\end{equation}
On the other hand, in Section \ref{sec; examples}, we prove that
\[
A(\D) \subsetneqq Z^\infty(\D).
\]
We do not know the measure of the set $H^\infty(\D) \setminus Z^\infty(\D)$. However, in the category of functions in $Z^\infty(\D)$, outer functions represent singly generated $M$-ideals in $H^\infty(\D)$ (see Theorem \ref{main1}): 

\begin{theorem}
Let $f\in Z^\infty(\D)$. Then $I_{H^\infty(\D)}(f)$ is an $M$-ideal in $H^\infty(\D)$ if and only if $f$ is an outer function.
\end{theorem}

Moreover, in the case of outer functions in $Z^\infty(\D)$, the corresponding singly generated $M$-ideals are explicit: If $f\in Z^\infty(\D)$ is an outer function, then (see Corollary \ref{vanishing set})
\[
\cli(f) = \{g \in H^\infty(\D):  Z_{\T}(f) \subseteq Z_{\T}(g) \text{ and } g \text{ has a continuous extension to }Z_{\T}(f)\}.
\]
In Example \ref{examp: ideal not in Z}, we present an $M$-ideal $I_{H^\infty(\D)}(f)$ in $H^\infty(\D)$ such that $f$ is an outer function and
\[
f \notin Z^\infty(\D).
\]
All the results and examples presented here suggest that the structure of $M$-ideals in $H^\infty(\D)$ is by far complicated, even at the level of principal ideals in $H^\infty(\D)$.

We have also revealed the structure of $M$-ideals in $H^\infty(\D)$ that are finitely generated by functions from $Z^\infty(\D)$. Given $\{f_1,\ldots, f_m\} \subseteq Z^\infty(\D)$, denote by $I_{H^\infty(\D)}(f_1, \ldots, f_m)$ the closed ideal generated by $\{f_1,\ldots, f_m\}$. In Theorem \ref{fg}, we prove: Let $\{f_1,\ldots, f_m\} \subseteq Z^\infty(\D)$. If $I_{H^\infty(\D)}(f_1, \ldots, f_m)$ is an $M$-ideal in $H^\infty(\D)$, then there exists an outer function $f \in Z^\infty(\D)$ such that
\[
I_{H^\infty(\D)}(f_1, \ldots, f_m) = I_{H^\infty(\D)}(f).
\]

The above result has some peculiarities within the framework of ideal theory. Indeed, this is reminiscent of the ascending chain condition in a ring. More specifically, the finitely generated ideal subjected to the $M$-ideal condition yields the subsequent terminating fact:
\[
I_{H^\infty(\D)} (f_1) \subseteq I_{H^\infty(\D)} (f_1, f_2) \subseteq \cdots \subseteq I_{H^\infty(\D)} (f_1, \cdots, f_m) = I_{H^\infty(\D)} (f).
\]
We further recall that a ring $R$ satisfies the ascending chain condition on ideals if and only if every ideal of $R$ is finitely generated. Note that rings that satisfy the ascending chain condition on ideals are also called \textit{Noetherian}.

Note that $H^\infty(\D)$ is also an example of a uniform algebra, and the classification of $M$-ideals in uniform algebras, which often involves maximal ideal spaces, is well known (see Theorem \ref{hd}). Unfortunately, since the maximal ideal space and the \v{S}ilov boundary of $H^\infty(\D)$ are hard to understand, uses of the general classification (which also includes peak sets and $p$-sets) on the uniform algebra $H^\infty(\D)$ only give abstract results.

We shall now delineate the structure of the subsequent sections of the paper. Section \ref{sec Prelim} provides an overview of the foundational concepts, recalls definitions, and conveys known results that will be used in the next sections. Section \ref{sec colossal} illustrates the extensive magnitude of $M$-ideals in $H^\infty(\D)$, while Section \ref{sec p sets} discusses $p$-sets from the standpoint of the space $H^\infty(\D)$. In Section \ref{sec prime}, we introduce the notion of analytic primes in $H^\infty(\D)$ and prove that proper $M$-ideals are analytic primes. A number of the results presented in this section pertain to Hilbert function space theory, such as Toeplitz operators, outer functions, and inner functions. Each result in this section applies to several variables. The primary emphasis of Section \ref{sec Principal ideal} is on principal ideals in $H^\infty(\D)$ that are generated by functions from $Z^\infty(\D)$. Section \ref{sec; examples} provides direct applications of results from previous sections and offers examples that illustrate the intricacy of the functions introduced earlier. Section \ref{sec; finitely gen} focuses on $M$-ideals generated by finite number of functions in $Z^\infty(\D)$.

\section{Preliminaries}\label{sec Prelim}

We begin with the formal definition of uniform algebras. A \textit{uniform algebra} on a compact Hausdroff space $K$ is a uniformly closed subalgebra of $C(K)$ which contains the constant functions and separates the points of $K$. Evidently, a uniform algebra is a commutative Banach algebra (under the uniform norm).

Now we set up some basic structures of commutative Banach algebras. Let $A$ be a commutative Banach algebra with unit. Let us denote by $\clm(A)$ the maximal ideal space of $A$. Since $A$ is unital, we know that $\clm(A) \subseteq S_{A^*}$, where given a Banach space $X$, we define the unit sphere of $X$ as
\[
S_X = \{x \in X:\|x\|_{X} = 1\}.
\]
Hence by the Banach–Alaoglu theorem and the fact that $\clm(A)$ is closed, it follows that $\clm(A)$ is a compact Hausdorff space in the weak-$*$ topology. If we denote by $C(\clm(A))$ the algebra of continuous functions on $\clm(A)$ with the supremum norm, then the \textit{Gelfand map}
\[
\Gamma : A \raro C(\clm(A)),
\]
is a contractive homomorphism of Banach algebras, where
\begin{align*}
\Gamma(f) = \hat{f},
\end{align*}
and
\[
\hat{f} (\vp) = \vp(f),
\]
for all $f\in A$ and $\vp \in \clm(A)$. A subset $C\subseteq K$ is called a \textit{boundary} for $A$ if
\[
\underset{x\in K}{\sup} |f(x)|=\underset{x\in C}{\max}|f(x)|,
\]
for all $f \in A$. The \textit{\v{S}ilov boundary} for $A$, denoted by $\partial_S A$, is the smallest closed boundary for $A$, that is \cite[page 173]{KH}
\[
\partial_S A = \bigcap \{C \subseteq K: C \text{ is a closed boundary for } A\} .
\]
Equivalently, $\partial_S A$ is the smallest closed subset of $\clm(A)$ on which every $\hat{f} \in \Gamma(A)$ attains its maximum modulus. Next, we turn to the unital commutative Banach algebra $H^\infty(\D)$ \cite[Chapter 10]{KH}. In this case
\begin{equation}\label{eqn: norm of infty}
\|f\|_\infty = \|\hat{f}\|_\infty = \sup_{\vp \in \clm(H^\infty(\D))} |\hat{f}(\vp)|,
\end{equation}
for all $f \in H^\infty(\D)$, that is, $\Gamma$ an isometric isomorphism from $H^\infty(\D)$ into $C(\clm(H^\infty(\D)))$. Evidently, $\Gamma$ is not onto. We use the Gelfand map $\Gamma : H^\infty(\D) \raro C(\mathcal{M}(H^\infty(\D)))$ to identify $H^\infty(\D)$ with $\widehat{H^\infty(\D)}$, where
\[
\widehat{H^\infty(\D)}:= \Gamma(H^\infty(\D)),
\]
Therefore, $\widehat{H^\infty(\D)}$, and hence an isometric copy of $H^\infty(\D)$, is a uniformly closed subalgebra of $C(\mathcal{M}(H^\infty(\D)))$. In particular, $\widehat{H^\infty(\D)}$ contains constant functions and separates points \cite{KH}.

Recall from \eqref{eqn: radial limit} that for a function $f \in H^\infty(\D)$, the radial limit extension of $f$ is denoted by $\tilde{f}$, where
\[
\tilde{f}(z) = \lim_{r \raro 1^-}f(rz) \qquad (z \in \T \;a.e.).
\]
This yields a natural way to identify $H^\infty(\D)$ with a closed subalgebra of $L^\infty(\T)$, where $L^\infty(\T)$ denotes the von Neumann algebra of all essentially bounded measurable complex-valued functions on $\T$. If $\widetilde{H^\infty(\T)}$ denotes the copy of $H^\infty(\D)$ in $L^\infty(\T)$, then
\[
\widetilde{H^\infty(\T)}=L^{\infty}(\T) \cap \widetilde{H^2(\T)},
\]
where $\widetilde{H^2(\T)} \subseteq L^2(\T)$ is the Hardy space on the unit circle $\mathbb{T}$ (again, via radial limits as in \eqref{eqn: radial limit}). In view of the above identification, we denote by $\tilde{f} \in \widetilde{H^\infty(\T)}$ the function corresponding to $f \in H^\infty(\D)$. In other words, we can represent $H^{\infty}(\D)$ as a uniformly closed subalgebra of $L^{\infty}(\T)$. This point of view is useful in identifying the \v{S}ilov boundary of $H^\infty(\D)$ in the sense that
\begin{equation}\label{eqn: tau M H}
\partial_S H^{\infty}(\D) = \tau(\clm(L^{\infty}(\T))),
\end{equation}
where the map $\tau:\clm(L^{\infty}(\T)) \rightarrow \clm(H^{\infty}(\D))$ defined by
\[
\tau (\vp) = \vp|_{H^\infty(\D)} \qquad (\vp\in \clm(L^{\infty}(\T))),
\]
is a homeomorphism \cite[page 174]{KH}. These tools will frequently be implemented in the next sections. In addition, we will see soon that $p$-sets are closely connected to the theory of $M$-ideals in $H^\infty(\D)$. We recall:

\begin{definition}\label{def peak set}
Let $A$ be a uniform algebra on a compact Hausdorff space $K$. A closed subset $C \subseteq K$ is said to be a peak set if there exists a function $f \in A$ such that
\[
f|_C \equiv 1,
\]
and
\[
|f(x)| < 1 \qquad (x \in X \setminus C).
\]
We frequently say that $f$ \textit{peaks on} $C$.
\end{definition}

Clearly, if $f\in A$ peaks on $C$, then
\[
\|f\| = 1.
\]
It is easy to see that a countable intersection of peak sets is a peak set \cite[page 208]{Gar}. In general, we define:

\begin{definition}
Let $A$ be a uniform algebra on a compact Hausdorff space $K$. A set $P\subset K$ is called a $p$-set of $A$ if $P$ is the intersection of a family of peak sets.
\end{definition}

Some of the well-known and general features of $p$-sets are as follows: (i) A $p$-set is a peak set if and only if it is a $G_\delta$-set \cite[Lemma 12.1]{Gamelin}. (ii) A countable union of $p$-sets is a $p$-set whenever the union is a closed set \cite[Corollary 12.8]{Gamelin}.

We conclude this introductory section by recalling the following characterizations of $M$-ideals in uniform algebras in terms of $p$-sets and bounded approximate units. This result will be frequently used in what follows.

\begin{theorem}\cite[Chapter V, Theorem 4.2]{HWW}\label{hd}
Let $A$ be a uniform algebra and let $J$ a closed subspace of $A$. Then the following conditions are equivalent:
\begin{enumerate}
\item[(i)] $J$ is an $M$-ideal of $A$.
\item[(ii)] $J$ is the annihilator of a p-set of $A$.
\item[(iii)] $J$ is an ideal of $A$ containing a bounded approximate unit.
\end{enumerate}
\end{theorem}

Recall that a bounded approximate unit in a commutative Banach algebra $A$ is a bounded net $\{x_i\}_{i \in \Lambda} \subseteq A$ such that
\[
\|x_i x - x \|_A \raro 0 \qquad (x \in A),
\]
along the net. In our analysis, we will mostly use the equivalence of (i) and (iii) in the preceding theorem. More specifically, we will enhance the practical relevance of this criterion by applying it to the specific context of uniform algebra $H^\infty(\D)$.

\section{$M$-ideals are colossal}\label{sec colossal}

The maximal ideal space of $H^\infty(\D)$ is notoriously intricate, making it difficult to fully understand the structure and properties of $H^\infty(\D)$. The objective of this brief section is to highlight the notoriety of the space of $M$-ideals, just analogous to the maximal ideal space of $H^\infty(\D)$. First, we recall a lemma concerning classification of $p$-sets \cite[V, Lemma 4.3]{HWW}. Denote by $e$ the unit of the given uniform algebra.

\begin{lemma}\label{123}
Let $A$ be an uniform algebra on a compact Hausdorff space $K$, and let $D$ be a closed subset of $K$. Then $D$ is a $p$-set for $A$ if and only if for each $\varepsilon>0$ and open set $U \supseteq D$, there exists $a \in A$ such that the following three conditions hold:
\begin{enumerate}
\item $\|e-a\| \leq 1 + \varepsilon$,
\item $a|_D=0$, and
\item $\Big|(e-a)|_{K \backslash U} \Big| <\varepsilon$.
\end{enumerate}
\end{lemma}

We also need to recall a Urysohn-type lemma for $H^\infty(\D)$ from \cite[Lemma 2.1]{BDSS}.

\begin{lemma}\label{1234}
Let $U$ be an open subset of $\clm(H^\infty(\D))$. Then for each $\varepsilon \in(0,1)$ and $\varphi_0 \in$ $U \cap \partial_S H^\infty(\D)$, there exists $f\in {H^{\infty}(\D)}$ such that
\begin{enumerate}
\item [(1)] $\|\hat{f}\|=\hat{f}\left(\varphi_0\right)=1$,
\item[(2)] $\sup \left\{|\hat{f}(\varphi)|: \varphi \in \partial_S H^\infty(\D)\backslash U\right\}<\varepsilon$, and
\item[(3)] $|\hat{f}(\varphi)|+(1-\varepsilon)|1-\hat{f}(\varphi)| \leq 1$ for all $\varphi \in \clm(H^\infty(\D))$.
\end{enumerate}
\end{lemma}

Now we turn to $M$-ideals in $H^\infty(\D)$. Fix $\psi\in \partial_S H^\infty(\D)$ and $\varepsilon > 0$. Then there exists a function $f\in H^\infty(\D)$ such that $f$ satisfies (1) to (3) of Lemma \ref{1234}. If we let
\[
g=1-f,
\]
then
\[
g|_{\{\psi\}} = 0,
\]
and
\[
\|1-g\| =1,
\]
and
\[
|(1-g)|_{\partial_S H^\infty(\D) \backslash U} \| < \varepsilon.
\]
This fulfils all requirements of Lemma \ref{123}, with $e$ equaling $1$ and $a$ equaling $g$. Consequently, it follows that $\{\psi\}$ is a $p$-set of $H^\infty(\D)$. Hence, by Theorem \ref{hd}, $\ker \psi$ is an $M$-ideal in $H^\infty(\D)$. Consequently, we have established the following result, which, in particular, states that the collection of $M$-ideals in $H^\infty(\D)$ is immense.
 
 
\begin{proposition}\label{prop M ideals larger}
Each maximum ideal in $H^\infty(\D)$ that corresponds to a complex homomorphism in $\partial_S H^\infty(\D)$ is an $M$-ideal in $H^\infty(\D)$.
\end{proposition}

On the contrary, it is easy to verify that each closed subset of $\mathbb{T}$ is a $p$-set for $C(\mathbb{T})$, where $C(\mathbb{T})$ denotes the space of all continuous functions on the unit circle $\mathbb{T}$. As a result, $M$-ideals in $C(\mathbb{T})$ are ideals of functions that vanish on some closed subset of $\mathbb{T}$. Observe that $C(\mathbb{T})$ is a commutative $C^*$-algebra.

On the other hand, the disc algebra $A(\D)$ is a commutative Banach algebra that is also a uniform algebra on $\overline{\D}$. At this point, we recall the Glicksberg peak set theorem \cite[p. 58]{Gamelin}: Let $A$ be a uniform algebra on a compact Hausdorff space $K$, and let $D$ be a closed subset of $K$. Then $D$ is a $p$-set if and only if $\mu_{D} \in A^\perp$ for all measure $\mu \in A^\perp$. That is, $D$ is a $p$-set if and only if $\mu|_E$ is orthogonal to $A$ for all measures $\mu$ orthogonal to $A$.

Returning to the uniform algebra $A(\D)$, we first observe in view of the F. and M. Reisz theorem, that any measure orthogonal to $A(\D)$ is absolutely continuous with respect to the Lebesgue measure. Hence the $p$-sets of $A(\D)$ are precisely the closed subsets of $\mathbb{T}$ having Lebesgue measure zero. Therefore, a closed subspace $C \subseteq A(\D)$ is an $M$-ideals in $A(\D)$ if and only if there exits a subset $D\subset \mathbb{T}$ of Lebesgue measure $0$ such that
\[
C = \{f\in A(\D): f|_D = 0 \}.
\]
Clearly, the structure of $M$-ideals in $A(\D)$ is simple. On the contrary, the above proposition clearly suggests that the situation of representing $M$-ideals in $H^\infty(\D)$ is a complex problem. Our results and methodology in this paper will also emphasise this characteristic.

\section{$p$-sets}\label{sec p sets}

According to Theorem \ref{hd}, the analysis of $M$-ideals in $H^\infty(\D)$ is comparable to the analysis of $p$-sets. The theory of $p$-sets presents an equally challenging problem for the space $H^\infty(\D)$, just like $M$-ideals. Within this section, our objective is to pick up a few significant characteristics of $p$-sets of $H^\infty(\D)$.

We start with representations of peak sets of $H^\infty(\D)$. For $\alpha\in \mathbb{T}$ and $f \in S_{H^\infty(\D)}$, define
\[
\mathscr{M}^f_\alpha = \{\vp \in \mathscr{M}(H^\infty(\D)): \hat{f}(\vp) = \alpha\}.
\]
It now follows from the definition itself that the peak sets of $H^\infty(\D)$ admit the form $\mathscr{M}^f_\alpha$ for some $\alpha\in \mathbb{T}$ and $f \in S_{H^\infty(\D)}$. If $f = z \in S_{H^\infty(\D)}$, then we simply write
\[
\mathscr{M}_\alpha = \mathscr{M}^f_\alpha \qquad (\alpha\in \mathbb{T}).
\]
Therefore, simple examples of peak sets in $H^\infty(\D)$ includes
\[
\mathscr{M}_\alpha = \{\vp\in \clm(H^\infty(\D)): \hat{z}(\vp) = \alpha\} \qquad (\alpha\in \mathbb{T}).
\]
Now we turn to $p$-sets and prove that a $p$-set of $H^\infty(\D)$ must meet the \v{S}ilov boundary $\partial_S H^{\infty}(\D)$.

\begin{theorem}\label{non}
Let $Q$ be a nonempty subset of $\clm(H^\infty(\D)) \setminus \partial_S H^{\infty}(\D)$. Then $Q$ is not a $p$-set of $H^\infty(\D)$.
\end{theorem}

\begin{proof}
Let $Q$ is a non-empty subset of $\clm(H^\infty(\D))$. Let
\[
P=Q\cap \partial_S H^{\infty}(\D).
\]
Assume that $Q$ is $p$-set of $H^\infty(\D)$. We claim that $P\neq \emptyset$. To this end, suppose $\{f_i\}_{i\in I} \subseteq S_{H^\infty(\D)}$ is the family of functions corresponding to the $p$-set $Q$, that is
\[
Q=\bigcap_{i \in I}\cam_1^{f_i}.
\]
Since $P=Q\cap \partial_S H^{\infty}(\D)$, we have
\[
P = \bigcap_{i \in I}\cam_1^{f_i}\bigcap \partial_S H^{\infty}(\D).
\]
Observe that $\clm(H^\infty(\D))$ is compact and $\cam_1^{f_i}$ is closed for each $i$. Therefore, if the collection
\[
\{\cam_1^{f_i}\cap \partial_S H^{\infty}(\D)\}_{i\in I},
\]
has finite intersection property (FIP), then one can conclude that $P\neq \emptyset$. To prove the FIP, for each finite collection $\{f_1, \ldots, f_n\} \subset S_{H^\infty(\D)}$, define
\[
\cam^{f_{i_1}, f_{i_2}, \ldots, f_{i_n}} = \bigcap_{j=1}^n\cam_1^{f_{i_j}}.
\]
As $\cam^{f_{i_1}, f_{i_2}, \ldots, f_{i_n}}$ is a finite intersection of peak sets, there exists a function $g \in S_{H^\infty(\D)}$ such that
\[
\cam_1^{g} = \cam^{f_{i_1}, f_{i_2}, \ldots, f_{i_n}}.
\]
Since $\|g\|=1$ and $\partial_S H^{\infty}(\D)$ is the \v{S}ilov boundary of $H^\infty(\D)$, there is at least one $\psi\in \partial_S H^{\infty}(\D)$ such that
\[
\psi(g)=1.
\]
This clearly implies that
\[
\psi\in \cam_1^{g}\cap \partial_S H^{\infty}(\D),
\]
and hence
\[
\begin{split}
\cam_1^{g}\cap \partial_S H^{\infty}(\D) & = \cam^{f_{i_1}, f_{i_2}, \ldots, f_{i_n}}\cap \partial_S H^{\infty}
\\
& = \cap_{j=1}^n\cam_1^{f_{i_j}}\cap \partial_S H^{\infty}(\D)
\\
& \neq \emptyset.
\end{split}
\]
Therefore, the collection $\{\cam_1^{f_i}\cap \partial_S H^{\infty}(\D)\}_{i\in I}$ satisfies the FIP, from which one concludes that $P\neq \emptyset$.
\end{proof}

Therefore, we have the following: If $Q$ is a $p$-set of $H^{\infty}(\D)$, then
\[
Q \cap \partial_S H^{\infty}(\D) \neq \emptyset.
\]
For each $\alpha\in \D$, denote by $ev_\alpha$ the evaluation functional at $\alpha$, that is
\[
ev_\alpha (f) = f(\alpha) \qquad (f \in H^\infty(\D)).
\]
The following is a consequence of Theorem \ref{non}:
 
\begin{corollary}\label{cor M ideal ev}
Let $\alpha\in \D$. Then $\{\alpha\}$ is not a $p$-set of $H^\infty(\D)$. In particular, the maximal ideal $\ker \text{ev}_\alpha$ is not an $M$-ideal in $H^\infty(\D)$.
\end{corollary}

Therefore, any subset of the unit disc $\D$ is not a $p$-set of $H^\infty(\D)$. On the other hand, the equivalence of (i) and (ii) in Theorem \ref{hd} says that the $M$-ideals in a uniform algebra $A$ are precisely of the form
\begin{equation}\label{eqn JP}
J_P=\{f\in A : f|_P =0\},
\end{equation}
where $P\subseteq K$ is a $p$-set of $A$. In general, we use $J_P$ to denote the set of all functions from the uniform algebra under consideration that vanish on a set $P \subseteq K$.

Typically, $M$-ideals in Banach spaces do not have a direct correlation (but some analogy; see the remark in the first paragraph of Section \ref{sec Intro}) with the notion of ideals in rings. Nevertheless, the above representations of $M$-ideals promptly reveals that this is not the case for uniform algebras.

We have the following as a corollary to Theorem \ref{hd}:

\begin{corollary}\label{cor M are ideals}
Let $A$ be a uniform algebra. Then $M$-ideals in $A$ are ideals in the ring $A$.
\end{corollary}

Considering an $M$-ideal $J$ in $H^\infty(\D)$, we will now establish a connection with two possible $p$-set representations of $H^\infty(\D)$ corresponding to $J$. Establishing such a link is inescapable from the perspective of two representations of the uniform algebra $H^\infty(\D)$ over the compact sets $\clm(H^\infty(\D))$ and $\partial_S H^{\infty}(\D)$. To be more precise, if $J$ is an $M$-ideal in $H^\infty(\D)$, then according to the representation \eqref{eqn JP}, there are two subsets $P \subseteq \clm(H^\infty(\D))$ and $Q \subseteq \partial_S H^{\infty}(\D)$ such that
\[
J = J_P = J_{Q}.
\]
In the following, we aim to establish a correlation between $P$ and $Q$. Part of the proof follows the lines of Hoffman \cite[page 187]{KH}.

\begin{theorem}
Let $J$ be an $M$-ideal in $H^\infty(\D)$. Suppose
\[
J = J_P,
\]
where $P \subseteq \clm(H^\infty(\D))$ is the corresponding $p$-set. Then
\[
J = J_{P \cap \partial_S H^{\infty}(\D)}.
\]
\end{theorem}
\begin{proof}
Given the representing $p$-set $P \subseteq \clm(H^\infty(\D))$ of the $M$-ideal $J$, we define
\[
Q = P \cap \partial_S H^{\infty}(\D).
\]
By Theorem \ref{non}, we know that $Q\neq \emptyset$. Let $\{f_i\}_{i\in I} \subseteq S_{H^\infty(\D)}$ be the family of functions corresponding to the $p$-set $P$. Then (see the notation preceding Theorem \ref{non})
\[
P = \bigcap_{i \in I} \cam_1^{f_i}.
\]
Fix an $i\in I$, and pick $\psi\in \clm_1^{{f_i}}$. Denote by $m_\psi$ the unique representing measure of $\psi$ on $\partial_S H^{\infty}(\D)$. We claim that $m_\psi$ is supported on $\clm_1^{{f_i}}\cap \partial_S H^{\infty}(\D)$. To see this, first we set
\[
h=\frac{1+f_i}{2}.
\]
For each $n \geq 1$, we observe that
\[
\begin{split}
\int_{\partial_S H^{\infty}(\D)} \hat{h}^n dm_\psi = \psi(h^n) =1.
\end{split}
\]
On the other hand, the sequence of functions $\{\hat{h}^n\}_{n \geq 1}$ is bounded and
\[
\hat{h}^n \longrightarrow \chi_{\clm_1^{f_i}},
\]
pointwise, where $\chi_{\clm_1^{f_i}}$ denotes the indicator function of $\clm_1^{f_i}$. By the dominated convergence theorem, we have
\[
1 = \lim_n\int_{\partial_S H^{\infty}(\D)} \hat{h}^n dm_\psi = \int_{\partial_S H^{\infty}(\D)} \hat{\chi}_{\clm_1^{f_i}} dm_\psi = \int_{\clm_1^{f_i}\cap \partial_S H^{\infty}(\D)} dm_\psi,
\]
and hence the measure $m_\psi$ is supported on ${\clm_1^{f_i}\cap \partial_S H^{\infty}(\D)}$, completing the proof of the claim. Since $P = \bigcap_{i\in I}\clm_1^{f_i}$, it follows that each $\psi$ in $P$ is supported on $\clm_1^{{f_i}}\cap \partial_S H^{\infty}(\D)$ for all $i \in I$. Hence the minimal support of $\psi$ is contained in
\[
\bigcap_{i\in I}\left(\clm_1^{f_i}\cap \partial_S H^{\infty}(\D) \right) = P \cap \partial_S H^{\infty}(\D) = Q.
\]
Therefore, for any $f\in J_{P\cap \partial_S H^{\infty}(\D)} = J_Q$ and $\psi \in P$, we have
\[
\psi(f) = \int_{P \cap\partial_S H^{\infty}(\D)}\hat{f} dm_\psi = 0,
\]
and hence $J_Q\subseteq J_P$. Since $Q\subseteq P$, it also follows that $J_Q\subseteq J_P$, and consequently
\[
J_P=J_Q,
\]
which completes the proof of the theorem.
\end{proof}

\section{Analytic primes}\label{sec prime}

This section will emphasise some significant properties of $M$-ideals in $H^\infty(\D)$. We do this in the generality of the polydisc $\D^n$ in $\mathbb{C}^n$, $n \geq 1$. Define 
\[
H^\infty(\D^n) = \{f \in \text{Hol}(\D^n): \|f \|:= \sup_{z \in \D^n} |f(z)| < \infty\}.
\]
By the same argumentation, $H^\infty(\D^n)$ is also a uniform algebra, and so Theorem \ref{hd} applies to $H^\infty(\D^n)$.

We stated at the very beginning of Section \ref{sec Intro} that the notion of $M$-ideals was proposed in \cite{AE} as a generalization of two-sided ideals in Banach spaces. Our key result in this section is yet another algebraic property of $M$-ideals. Recall that an ideal $P$ in a commutative ring $R$ is called prime if $P \neq R$ and if $a$ and $b$ are two elements of $R$ such that
\[
ab \in P,
\]
then either $a \in P$ or $b \in P$. This motivates the analytic definition of prime ideals in $H^\infty(\D^n)$ (see Definition \ref{def intro prime}): Let $J$ be a proper nontrivial closed subspace of $H^\infty(\D^n)$. We call $J$ an analytic prime as long as the following property holds true: If $f$ and $g$ are two elements of $H^\infty(\D^n)$ such that $fg \in J$ and
\[
|\tilde{f}(z)| > \delta \qquad (z \in \T^n),
\]
for some $\delta > 0$, then $g \in J$.

As in the $n=1$ case, functions in $H^\infty(\D^n)$ also admit boundary values a.e. on the $n$-torus $\T^n$ \cite{WR book}. Note that the notion of analytic primes has the potential to extend to a broad class of uniform algebras (although one may need to find a suitable replacement for $\T^n$). It is also clear that the concept of an analytic prime is more potent than prime ideals in $H^\infty(\D^n)$ (or a uniform algebra). This also pertains to the structure of $M$-ideals in $H^\infty(\D^n)$. Specifically:

\begin{theorem}\label{prop1}
Proper nontrivial $M$-ideals in $H^\infty(\mathbb{D}^n)$ are analytic primes.
\end{theorem}

\begin{proof}
Let $J$ be a proper nontrivial $M$-ideal in $H^\infty(\mathbb{D}^n)$. Since $J$ is an $M$-ideal and $H^\infty(\mathbb{D}^n)$ is a uniform algebra, it follows that $J$ contains a bounded approximate unit $\{f_\lambda\}_ {\lambda \in \Lambda}$. Let $f, g \in H^\infty(\D^n)$, and suppose $f g \in J$. Suppose there is a $\delta > 0$ such that
\[
|\tilde{f}(z)| > \delta \qquad (z \in \T^n).
\]
Fix $\varepsilon> 0$. There exists $\lambda_\varepsilon \in \Lambda$ such that
\[
\|f g f_\lambda - f g \| \leq \varepsilon,
\]
for all $\lambda \geq \lambda_\varepsilon$. Since this is the case, there exists a sequence $\{\lambda_m\} \subseteq \Lambda$ such that
\[
\|f g f_{\lambda_m}- f g\| \leq \frac{1}{m},
\]
for all $m \geq 1$. In view of the identification of functions via radial limits, we now observe that, for functions in $H^\infty(\D^n)$, the supremum on the boundary is sufficient to take into account. Section \ref{sec Prelim} describes the identification for this $n=1$ scenario, whereas the $n>1$ case works similarly. Fix an integer $m \geq 1$. For all $z \in \T^n$, we have
\[
\begin{split}
\delta | \tilde{g}(z) \tilde{f}_{\lambda_m}(z) - \tilde{g}(z)| \leq |\tilde{f}(z)||\tilde{g}(z) \tilde{f}_{\lambda_m}(z) - \tilde{g}(z)| = |\tilde{f}(z) \tilde{g}(z) \tilde{f}_{\lambda_m}(z)- \tilde{f}(z) \tilde{g}(z)|.
\end{split}
\]
Taking the supremum over all $z \in \T^n$ gives
\[
\delta \| \tilde{g} \tilde{f}_{\lambda_m}- \tilde{g} \| \leq | | \tilde{f}\tilde{g}  \tilde{f}_{\lambda_m}- \tilde{f}\tilde{g} \|,
\]
and hence
\[
\|\tilde{g} \tilde{f}_{\lambda_m}- \tilde{g}\| \leq \frac{1}{\delta}\frac{1}{m},
\]
for all $m \geq 1$. Since
\[
\tilde{g} \tilde{f}_{\lambda_m} \in J,
\]
for all $m \geq 1$, it follows that $g \in J$. Thus, the proof of the theorem is concluded.
\end{proof}

The analytic prime property of $M$-ideals in $H^\infty(\D)$ will be frequently used in what follows. Recall that a function $\upsilon\in H^\infty(\D^n)$ is said to be \textit{inner} if
\[
|\tilde\upsilon(z)| = 1,
\]
for all $z \in \T^n$ a.e. We now present one implication from the preceding result, which is of independent relevance: An $M$-ideal never intersects the multiplicative set of inner functions:

\begin{corollary}\label{inn}
Let $J$ be a nontrivial and proper $M$-ideal in $H^\infty(\D^n)$. Then no inner function belongs to $J$.
\end{corollary}
\begin{proof}
Recall from Corollary \ref{cor M are ideals} that $J$ is a (two sided) ideal of $H^\infty(\D^n)$. Suppose that there exists an inner function $\upsilon \in H^\infty(\D^n)$ such that $\upsilon \in J$. Then, by writing
\[
\upsilon = 1 \times \upsilon,
\]
we conclude from Theorem \ref{prop1} that $1 \in J$. This is contrary to the fact that $J$ is proper.
\end{proof}

There is an alternate proof of the above: Recall from \eqref{eqn JP} that there exists $P \subseteq \clm(H^\infty(\D^n))$ such that $J = J_P$. Now Theorem \ref{non} implies that $P$ must meet the \v{S}ilov boundary of $H^\infty(\D^n)$. However, inner functions do not vanish on the \v{S}ilov boundary \cite[page 179]{KH}.

Given that inner functions are extreme points of the closed unit ball of $H^\infty(\D)$, the above corollary naturally raises the question of whether the same conclusion also holds for all extreme points. However, this is not the case; see the example at the end of Section \ref{sec Principal ideal}.

Recall that singular inner functions are those inner functions that are zero-free. As a visual, for each $\alpha \in \T$, the function
\[
\upsilon_\alpha(z):= \exp\Big(\frac{z+\alpha}{z-\alpha}\Big) \qquad (z \in \D),
\]
is a singular inner function in $H^\infty(\D)$. We illustrate the above result through the following example:

\begin{example}
Consider the closed ideal $J$ in $H^\infty(\D)$, where
\[
J = \{f \in H^\infty(\D): \lim_{r\rightarrow 1} f(r)=0\}.
\]
Clearly, $\upsilon_1 \in J$. Hence, by Corollary \ref{inn}, it follows that $J$ is not an $M$-ideal in $H^\infty(\D)$.
\end{example}

Now we present another application of Theorem \ref{prop1}. This time, it will be applied to the Hardy space $H^2(\D^n)$. Recall that \cite{WR book}
\[
H^2(\D^n) = \Big\{f \in \text{Hol}(\D^n): \|f\|_2:= \Big(\sup_{0<r<1} \int_{\T^n} |f(rz)|^2 d\mu(z)\Big)^{\frac{1}{2}} < \infty\Big\},
\]
where $d\mu$ denotes the normalized Lebesgue measure on $\T^n$ and $rz =(rz_1, \ldots, z_n)$. In view of the radial limits, we have the following isometric embedding:
\[
H^2(\D^n) \hookrightarrow L^2(\T^n).
\]
Moreover, it is well-known that
\begin{equation}\label{eqn dens of poly}
\overline{\mathbb{C}[z_1, \ldots, z_n]}^{H^2(\D^n)} = H^2(\D^n).
\end{equation}
We also need some operator theoretic concepts that are classic in nature. Given a function $\vp \in H^\infty(\D^n)$, define the \textit{Toeplitz operator} $T_\vp$ in $\clb(H^2(\D^n))$ by
\[
T_\vp f = \vp f \qquad (f \in H^2(\D^n)).
\]
The symbols with the independent variables $\{z_1, \ldots, z_n\}$ are known as distinguished Toeplitz operators. In this case, we have
\[
T_{z_i} f = z_i f \qquad (f \in H^2(\D^n)),
\]
for all $i=1, \ldots, n$. The following says, in particular, that the density assertion in \eqref{eqn dens of poly} holds for $M$-ideals in $H^\infty(\D^n)$.

\begin{theorem}\label{thm M is dense}
If $J$ is a nontrivial $M$-ideal in $H^\infty(\mathbb{D}^n)$, then
\[
\overline{J}^{H^2(\D^n)} = H^2(\mathbb{D}^n).
\]
\end{theorem}
\begin{proof}
By \eqref{eqn JP}, there exists $p$-set $P$ of $H^\infty(\D^n)$ such that 
\[
J = \{\vp \in H^\infty(\D): \vp|_P = 0\}.
\]
From here, it is easy to conclude that $J$ is invariant under $z_i$, $i=1, \ldots, n$. Therefore, $\overline{J}^{H^2(\D^n)}$ is also invariant under $z_i$, $i=1, \ldots, n$. It is then sufficient to show that $1$ is in $\overline{J}^{H^2(\D^n)}$. To accomplish this, fix a bounded approximate unit $\{\vp_\lambda\}_ {\lambda \in \Lambda}$ in $J$. As in the proof of Theorem \ref{prop1}, for a fixed $\vp \in J$, there exists a sequence $ \{\lambda_m\} \subseteq \Lambda$ such that
\begin{equation}\label{eq1}
\|\vp \vp_{\lambda_m}-  \vp\|_{H^\infty(\D^n)} \leq \frac{1}{m},
\end{equation}
for all $m \geq 1$. Note that $\{\vp_{\lambda_m}\}$ is a bounded sequence in $H^2(\mathbb{D}^n)$. Consequently, $\{\vp_{\lambda_m}\}$ has a weak convergent subsequence, and hence, there exists $f\in H^2(\mathbb{D}^n)$ such that
\[
\vp_{\lambda_{m_k}} \overset{w}\longrightarrow f,
\]
in $H^2(\D^n)$. Since the Toeplitz operator $T_\vp$ on $H^2(\mathbb{D}^n)$ is bounded, we obtain
\[
\vp \vp_{\lambda_{m_k}} \overset{w}\longrightarrow \vp f,
\]
in $H^2(\D^n)$. Also by \eqref{eq1}, we have
\[
\vp \vp_{\lambda_{m_k}} \overset{w}\longrightarrow \vp,
\]
in $H^2(\D^n)$. Here we are using the general fact that $H^\infty(\D^n)$ is contractively embedded in $H^2(\D^n)$. Therefore
\[
\vp f = \vp,
\]
equivalently
\[
f  - 1 \in \ker T_\vp.
\]
However, since $\vp$ is an analytic function, we know that
\[
\ker T_\vp = \{0\}.
\]
This implies $f = 1$, and completes the proof of the theorem, as weak closure and norm closure are identical for convex sets.
\end{proof}

Recall that for a given $f \in H^\infty(\D^n)$, we denote by $I_{H^\infty(\D^n)}(f)$ the closed principal ideal generated by $f$ in $H^\infty(\D^n)$ (see\eqref{eq prin ideal}):
\[
I_{H^\infty(\D^n)}(f) = \overline{\{p f: p \in \mathbb{C}[z_1, \ldots, z_n]\}}^{H^\infty(\D^n)}.
\]

The following is a straight application of the above theorem, which establishes a clear link between $M$-ideals and classical operators such as Toeplitz operators.

\begin{corollary}\label{cor ker of TO}
Let $\vp \in H^\infty(\mathbb{D}^n)$. If $I_{H^\infty(\D^n)}(\vp)$ is an $M$-ideal in $H^\infty(\mathbb{D}^n)$, then
\[
\ker T_\vp^* = \{0\}.
\]
\end{corollary}
\begin{proof}
Note that $\ker T_\vp^* = (\text{ran} T_\vp)^{\perp}$. But
\[
\begin{split}
\overline{\text{ran} T_\vp}^{H^2(\D^n)} & = \overline{I_{H^\infty(\D^n)}(\vp)}^{H^2(\D^n)} = H^2(\D^n).
\end{split}
\]
The result now follows immediately.
\end{proof}

In the first part of the proof of Theorem \ref{thm M is dense}, we verified that an $M$-ideal in $H^\infty(\D^n)$ is an invariant subspace of $H^2(\D^n)$. In this context, a subspace $\cls$ of $H^2(\D^n)$ is considered \textit{invariant} if
\[
z_i \cls \subseteq \cls \qquad (i=1, \ldots, n).
\]

\begin{corollary}
Let $J$ be a nontrivial closed ideal in $H^\infty(\mathbb{D}^n)$. If there exists a proper closed invariant subspace $\cls$ of $H^2(\D^n)$ such that
\[
J \subseteq \cls,
\]
then $J$ is not an $M$-ideal in $H^\infty(\D^n)$.
\end{corollary}
\begin{proof}
Since $\cls$ is a proper closed subspace of $H^2(\D^n)$, by Theorem \ref{thm M is dense}, it follows that $J$ is not an $M$-ideal; otherwise
\[
H^2(\D^n) \supsetneqq \cls \supseteq \overline{J}^{H^2(\D^n)} = H^2(\D^n),
\]
a contradiction.
\end{proof}

Given a function $f \in H^2(\D^n)$, define
\[
I_{H^2(\D^n)}(f) = \overline{\text{span}\{p f: p \in \mathbb{C}[z_1, \ldots, z_n]\}}^{H^2(\D^n)}.
\]
In other words, $I_{H^2(\D^n)}(f)$ is the smallest closed invariant subspace of $H^2(\D^n)$ containing $f$.

\begin{corollary}\label{c2}
Let $f \in H^\infty(\mathbb{D}^n)$. If $I_{H^\infty(\D^n)}(f)$ is an $M$-ideal in $H^\infty(\D^n)$, then
\[
I_{H^2(\D^n)}(f) = H^2(\D^n).
\]
\end{corollary}
\begin{proof}
First, we observe that
\[
I_{H^2(\D^n)}(f) = \overline{I_{H^\infty(\D^n)}(f)}^{H^2(\D^n)}.
\]
Since $I_{H^\infty(\D^n)}(f)$ is an $M$-ideal in $H^\infty(\D^n)$, the result directly follows from Theorem \ref{thm M is dense}.
\end{proof}

If $n = 1$, then a function $f \in H^\infty(\D)$ is refer to as an \textit{outer function} if
\begin{equation}\label{eqn outer fn I}
I_{H^2(\D)}(f) = H^2(\D).
\end{equation}
Therefore, Corollary \ref{c2} implies that if $I_{H^\infty(\D)}(f)$ is an $M$-ideal, then $f$ is an outer function. 

Note that the outer functions on the disc are zero-free. The notion of outer functions in higher variables is, however, different from the above (see Rudin \cite[page 72]{Rud 57}).

\section{Principal $M$-ideals}\label{sec Principal ideal}

From now on, we shall limit ourselves to single variable as a result of the unavailability of tools in several variables that are to be utilised in the subsequent computations. We begin with Corollary \ref{c2} of the previous section. It raises the question of a converse direction: Is the ideal generated by an outer function in $H^\infty(\D)$ always an $M$-ideal in $H^\infty(\D)$?

The main goal of this section is to prove that for functions in $Z^\infty(\D)$, the converse is true. Recall from Definition \ref{def ZD} that $Z^\infty(\D)$ is the set of all functions $f$ in $H^\infty(\D)$ that have a continuous extension to $\D\cup Z_\T(f)$, where
\[
Z_\T(f) = \{z\in \T: z\in Z_\T(f_O)\} \qquad (f \in H^\infty(\D)),
\]
and $f_O$ denotes the outer factor of the inner-outer factorization of $f \in H^\infty(\D)$. Observe that there may exist outer functions $f \in H^\infty(\D)$ and points $z\in Z_\T(f)$ such that $\tilde{f}(z)$ does not exist. In view of the above definition, the significance of considering functions in $Z^\infty(\D)$ is that this issue does not arise: If $f\in Z^\infty(\D)$ and $z \in Z_\T(f)$, then for any sequence $\{z_n\}$ converging to $z$, we have $f(z_n) \raro 0$. In particular, the radial limit function satisfies the following:
\[
\tilde{f}(z) = \lim_{r \raro 1^-}f(rz) =0.
\]

We recall representations of outer functions, which will play a crucial role in our analysis. For each $z \in \D$ and $\theta \in [-\pi, \pi]$, we define (in this context, also recall the Szeg\"{o} kernel on $\D$)
\[
S(z,\theta) = \frac{e^{i\theta}+z}{e^{i\theta}-z}.
\]
It is well known that a function $f \in H^\infty(\D)$ (or $f \in H^2(\D)$) is outer if and only if (see \cite[page 62]{KH}) there exists a real-valued integrable function $k$ on $[-\pi, \pi]$ such that
\begin{equation}\label{eqn outer fn}
f(z)=exp\left(\frac{1}{2\pi}\int_{-\pi}^\pi S(z,\theta)k(\theta)d\theta \right)\qquad (z \in \D).
\end{equation}
Moreover, in this case
\[
k(\theta)=\log|f(e^{i\theta})|,
\]
for all $\theta \in [-\pi, \pi]$ a.e. In what follows, we will assume $Z_\T(f)\neq \emptyset$, otherwise $f$ is invertible and
\[
I_{H^\infty(\D)}(f)=H^\infty(\D).
\]
We need the following lemma:

\begin{lemma}\label{lemma: outer functions}
Let $f \in Z^\infty(\D)$ be an outer function. Then there exists a sequence of open sets $\{A_m\}_{m \geq 1}$ in $\T$ such that the following conditions hold:
\begin{enumerate}
\item $Z_{\T}(f) \subseteq A_m$ for all $m \geq 1$.
\item There exists an open set $U_m \subseteq \T$ such that $\overline{A_{m+1}}\subseteq U_m\subseteq A_m$ for all $m \geq 1$.
\item $\mu{(A_m)}\searrow 0$ as $m \raro \infty$.
\item $|\tilde{f}| \leq e^{-m}$ on $A_m$ for all $m \geq 1$.
\end{enumerate}
\end{lemma}

\begin{proof}
Consider the representation of the outer function $f$ as in \eqref{eqn outer fn}:
\[
f(z)=exp\left(\frac{1}{2\pi}\int_{-\pi}^\pi S(z,\theta)k(\theta)d\theta \right) \qquad (z \in \D).
\]
By the outer regularity of the Lebesgue measure, there exists a sequence of open sets $\{C_m\}$ in $\mathbb{T}$ such that
\[
Z_{\T}(f) \subseteq C_m,
\]
and
\[
{C_{m+1}}\subseteq C_m,
\]
and
\[
\mu(C_m)\rightarrow \mu(Z_{\T}(f)) = 0.
\]
For all $m \geq 1$. Choose a sequence of positive real numbers $\{\varepsilon_m\}_{m \geq 1}$ such that
\[
 \varepsilon_m<e^{-m} \qquad (m \geq 1).
\]
Since $f$ is continuously extendable to $Z_{\T}(f)$, corresponding to $\varepsilon_1 >0$ and $\alpha \in Z_{\T}(f)$, there exists a $\delta_\alpha^1 > 0$ such that
\[
|f(z)-f(\alpha)| <\varepsilon_1,
\]
for all $z\in \mathbb{D}\cap B(\alpha, \delta_\alpha^1)$, and
\[
\overline{B(\alpha_j, \delta_{\alpha_j}^1)} \cap \T \subsetneqq C_1,
\]
where, for each $z_0 \in \mathbb{C}$ and $\delta > 0$, we write
\[
B(z_0, \delta) = \{z \in \mathbb{C}: |z - z_0| < \delta\}.
\]
Since $f(\alpha) = 0$, it follows that
\[
|f(z)|<\varepsilon_1,
\]
for all $z\in \mathbb{D}\cap B(\alpha, \delta_\alpha^1)$. As $Z_{\T}(f)$ is a compact set, there exist scalars $\{\alpha_j\}_{j=1}^{n_1}$ such that
\[
\{\alpha_j\}_{j=1}^{n_1} \subseteq Z_{\T}(f) \subseteq \bigcup_{j=1}^{n_1} B(\alpha_j, \delta_{\alpha_j}^1) .
\]
Define
\[
B_{\varepsilon_1} = \bigcup_{j=1}^{n_1} B(\alpha_j, \delta_{\alpha_j}^1),
\]
and
\[
A_1 = \T \cap B_{\varepsilon_1}.
\]
Now for $\varepsilon_2 >0$ and for each $\alpha \in Z_{\T}(f)$, there exists a $\delta_\alpha^2 > 0$ such that
\[
|f(z)| <\varepsilon_2,
\]
for all $z \in \D \cap B(\alpha, \delta_{\alpha}^2)$, and
\begin{equation}\label{clos}
\overline{B(\alpha, \delta_{\alpha}^2)} \subsetneqq B_{\varepsilon_1}, \overline{B(\alpha, \delta_{\alpha}^2)}\cap \T \subsetneqq A_1 \cap C_2
\end{equation}
Again, by compactness of $Z_{\T}(f)$, there exist scalars $\{\alpha_j\}_{j=1}^{n_2}$ such that
\[
\{\alpha_j\}_{j=1}^{n_2} \subseteq Z_{\T}(f) \subseteq \bigcup_{j=1}^{n_2} B(\alpha_j, \delta_{\alpha_j}^2).
\]
Similarly, define
\[
B_{\varepsilon_2} = \bigcup_{j=1}^{n_2} B(\alpha_j, \delta_{\alpha_j}^2),
\]
and
\[
A_2 = \T \cap B_{\varepsilon_2}.
\]
Clearly
\[
\overline{B}_{\varepsilon_2}\cap \mathbb{T} =  \bigcup_{j=1}^{n_2} \big( \overline{B(\alpha_j, \delta_{\alpha_j}^2)} \cap \mathbb{T}\big )\subsetneqq B_{\varepsilon_1}\cap \mathbb{T},
\]
where the strict inclusion follows from the inclusions pointed out in (\ref{clos}) and the fact that $B_{\varepsilon_1}$ is open. Therefore
\[
\overline{A}_{2} = \overline{B}_{\varepsilon_2} \cap \T \subsetneqq  {B_{\varepsilon_1}} \cap \T = A_{1}.
\]
Evidently, this process continues for $m >2$. So we get a sequence  $\{A_m\}_{m \geq 1}$ of open sets
such that:

\noindent (1) $Z_{\T}(f) \subseteq A_{m}$, $m\geq 1$, and

\noindent (2) $\overline{A_{m+1}}\subseteq U_m\subseteq A_m$ for all $m \geq 1$,

\noindent which yields conditions (1) and (2). Also
\[
0\leq \mu(A_m)\leq \mu(C_m) \rightarrow 0,
\]
implies (3). We now prove the remaining one, condition (4). Let $m\in \mathbb{N}$ be fixed.
Suppose $\tilde{f}(e^{it})$ exists for some $e^{it}\in A_m$ (that is, $\tilde{f}(e^{it}) \in \mathbb{C}$). Then there exists $\alpha \in Z_{\T}(f)$ such that
\[
e^{it} \in B(\alpha, \delta_\alpha^m).
\]
Since $B(\alpha, \delta_\alpha^m)$ is open, there exists $0<s<1$ such that
\[
re^{it}\in B(\alpha, \delta_\alpha^m)\cap \mathbb{D},
\]
for all $s< r < 1$. Therefore
\[
|f(re^{it})|<\varepsilon_m,
\]
for all $s\leq r<1$, and hence, the radial limit extension function $\tilde f$ satisfies the following property:
\[
|\tilde{f}(e^{it})|<\varepsilon_m.
\]
Since $e^{it}$ is arbitrary, we get
\[
\chi_{A_m}|\tilde{f}|<\varepsilon_m<e^{-m},
\]
which yields (4), and completes the proof of the lemma.
\end{proof}

Now we are ready to prove the main result of this section. Recall that for $f \in H^\infty(\D)$, we denote by $I_{H^\infty(\D)}(f)$ the closed ideal generated by $f$.

\begin{theorem} \label{main1}
Let $f\in Z^\infty(\D)$. Then $I_{H^\infty(\D)}(f)$ is an $M$-ideal in $H^\infty(\D)$ if and only if $f$ is an outer function.
\end{theorem}
\begin{proof}
The necessary part is already in Corollary \ref{c2} (along with the identity \eqref{eqn outer fn I}). For the sufficient part, assume that $f \in Z^\infty(\D)$ is an outer function. Set
\[
k(\theta):=\log|f(e^{i\theta})|,
\]
for $\theta \in [-\pi, \pi]$ a.e. Then $k$ is a real-valued integrable function, and
\[
f(z)=exp\left(\frac{1}{2\pi}\int_{-\pi}^\pi S(z,\theta)k(\theta)d\theta \right) \qquad (z \in \D),
\]
where $S(z,\theta) = \frac{e^{i\theta}+z}{e^{i\theta}-z}$ (see the discussion at the very start of this section). It is enough to build a bounded approximate unit in $I_{H^\infty(\D)}(f)$ (see Theorem \ref{hd}). We proceed as follows: By Lemma \ref{lemma: outer functions}, we construct a sequence of open sets $\{A_m\}_{m \geq 1}$ in $\T$ such that
\begin{enumerate}
\item $Z_{\T}(f) \subseteq A_m$ for all $m \geq  1$.
\item $\overline{A_{m+1}}\subseteq U_m\subseteq A_m$, $U_m$ open for all $m \geq 1$.
\item $\mu{(A_m)}\searrow 0$ as $m \rightarrow \infty$.
\item $|\tilde f| \leq e^{-m}$, or equivalently
\[
k(\theta) < - m,
\]
a.e. on $A_m$ for all $m \geq 1$.
\end{enumerate}
For each $m \geq 1$, define a real-valued integrable function $k_m$ on $[-\pi, \pi]$ as
\[
k_m = -\chi_{A_m^c} k.
\]
For each $m \geq 1$, we claim that $k_m$ is bounded above a.e. We prove this by contradiction. Fix $m \geq 1$. Assume, for each $N\in \mathbb{N}$, we have that
\[
\mu\left(\{\theta : {k_m(\theta)}>N\}\cap A_m^c\right) >0.
\]
But
\begin{align*}
\mu\left(|\tilde{f}|^{-1}(0, e^{-N})\cap A_{m}^c\right) & = \mu\left( \{\theta : {k(\theta)}<-N \} \cap A_m^c\right)
\\
& = \mu\left( \{\theta : {k_m(\theta)}> N \} \cap A_m^c\right)
\\
& >0.
\end{align*}
Because of the inner regularity of Lebesgue measure, there exists a compact set $B_N$ such that
\[
B_N\subseteq |\tilde{f}|^{-1}(0, e^{-N})\cap A_{m}^c,
\]
and $\mu(B_N)>0$. Since $B_N$ is compact for each $N$, we can find $z_N\in B_N$ and a neighbourhood $N_{z_N}$ of $z_N$ such that $\mu(N_{z_N})=e^{-N}$ and
\[
\mu(N_{z_N}\cap B_N)>0.
\]
To see this, one may consider a cover by a set of open intervals of radius $e^{-N}$ of $B_N$. Then, because $B_N$ is compact, there exists a finite subcover. Since
\[
\mu(B_N) > 0,
\]
there exists at least one interval, say $N_{z_N}$, in the finite sub cover that intersects $B_N$ with a positive measure. Take $z_N$ to be the centre of this interval $N_{z_N}$. Passing onto a subsequence if necessary we can assume that $\{z_N\}$ converges to $z$. Let $\varepsilon>0$ and let $N_z$ be a neighbourhood of $z$. Without any loss of generality, we assume that $N_z$ is centpurple at $z$. Then
\[
|z_N-z|< \mu(N_z)/4,
\]
for all $N>N_0$. Take $N$ sufficiently large so that $N>N_0$ and
\[
\mu(N_{z_N})<\mu(N_z)/4,
\]
and $e^{-N}<\varepsilon$. Then, for each $w\in N_{z_N}$, we have
\[
\begin{split}
|w-z| & \leq |w-z_N|+|z_N-z|
\\
& <\mu(N_{z_n})/4+\mu(N_z)/4
\\
& < \mu(N_z)/2.
\end{split}
\]
Therefore $N_{z_N}\subseteq N_z$, and hence $N_{z_N}\cap B_N\subseteq N_z$. Since $|\tilde{f}|<e^{-N}<\varepsilon$ a.e on $B_N$, it follows that
\[
\mu(\{|\tilde{f}|^{-1}(0, \varepsilon)\cap N_z\}) \geq \mu(N_{z_N}\cap B_{N_0})>0.
\]
This implies that $A_{m}^c$ has a zero of $f$, which is impossible because $Z_\T(f)\subset A_m$. This proves the claim, and subsequently, we define outer function $g_m$ by (recall the representations of outer functions from \eqref{eqn outer fn})
\[
g_m(z) = \operatorname{exp}\left( \frac{1}{2\pi}\int_{-\pi}^{\pi}S(z, \theta)k_m(\theta) d\theta\right) \qquad (z \in \D).
\]
At this point, we recall that the real part of $S(z, \theta)$ is the Poisson kernel $P(z, \theta)$ (see \cite[page 30]{Hoffman} for more details). If we write $z=re^{i\xi}$, then as $r \rightarrow 1$, we have
\[
\frac{1}{2\pi}\int_{-\pi}^{\pi}\Big[\operatorname{Re}S(re^{i\xi}, \theta)\Big] k_m(\theta)\, d\theta \rightarrow k_m(\xi),
\]
a.e. Now
\begin{align*}
|\tilde{g}_m(e^{i\xi})| & = \left |\lim_{r\rightarrow1}\operatorname{exp}\left(\frac{1}{2\pi}\int_{-\pi}^{\pi}S(re^{i\xi}, \theta)k_m(\theta)d\theta\right)\right|
\\
& = \lim_{r\rightarrow1}\operatorname{exp}\left(\frac{1}{2\pi} \int_{-\pi}^{\pi}P(re^{i\xi}, \theta)k_m(\theta) d\theta\right)
\\
& = \operatorname{exp}\left(\frac{1}{2\pi}\lim_{r\rightarrow1}\int_{-\pi}^{\pi}P(re^{i\xi},\theta) (-\chi_{A_m^c}(e^{i\theta})k(\theta))d\theta\right)
\\
& = \operatorname{exp}\left(-\chi_{A_m^c}(e^{i\xi})k(\xi)\right)
\\
& = \operatorname{exp}\Big(k_m(\xi)\Big),
\end{align*}
a.e. Since $k_m$ is bounded, it follows that $g_m\in H^\infty(\D)$. Again, we compute
\begin{align*}
(fg_m)(re^{i\xi}) & = \operatorname{exp}\left(\frac{1}{2\pi}\int_{-\pi}^{\pi}S(re^{i\xi}, \theta) k(\theta) d\theta\right) \times \operatorname{exp}\left( \frac{1}{2\pi}\int_{-\pi}^{\pi} S(re^{i\xi}, \theta)k_m(\theta) d\theta\right)
\\
& = \operatorname{exp}\left( \frac{1}{2\pi}\int_{-\pi}^{\pi} S(re^{i\xi}, \theta)k(\theta) d\theta\right) \times \operatorname{exp}\left( \frac{1}{2\pi}\int_{-\pi}^{\pi}S(re^{i\xi}, \theta)-\chi_{A_m^c}(e^{i\theta})k(\theta)d\theta\right)
\\
& = \operatorname{exp}\left( \frac{1}{2\pi}\int_{-\pi}^{\pi}S(re^{i\xi}, \theta)\left(k(\theta) -\chi_{A_m^c}(e^{i\theta}) k(\theta)\right) d\theta\right)
\\
& = \operatorname{exp}\left(\frac{1}{2\pi}\int_{-\pi}^{\pi} S(re^{i\xi}, \theta)\chi_{A_m}(e^{i\theta}) k(\theta) d\theta\right).
\end{align*}
We now claim that $\{fg_m\}_{m \geq 1}$ is a bounded approximate unit in $I_{H^\infty(\D)}(f)$. For fixed $m \geq 1$ and $e^{i\xi} \in \T$, we have
\begin{align*}
|\widetilde{(fg_m)}(e^{i\xi})| & = \left |\lim_{r\rightarrow1} \operatorname{exp}\left(\frac{1}{2\pi}\int_{-\pi}^{\pi}S(re^{i\xi}, \theta) \chi_{A_m}(e^{i\theta})k(\theta) d\theta\right)\right |
\\
& = \lim_{r\rightarrow1}\operatorname{exp}\left(\frac{1}{2\pi}\int_{-\pi}^{\pi} P(re^{i\xi}, \theta)\chi_{A_m}(e^{i\theta})k(\theta) d\theta\right)
\\
& =  \operatorname{exp}\left(\frac{1}{2\pi}\lim_{r\rightarrow1}\int_{-\pi}^{\pi} P(re^{i\xi}, \theta)\chi_{A_m}(e^{i\theta})k(\theta) d\theta\right)
\\
& = \operatorname{exp}\left(\chi_{A_m}(e^{i\xi}) k(\xi)\right),
\end{align*}
a.e. But, on one hand
\[
\chi_{A_m}(e^{i\xi})k(\xi) = 0,
\]
for all $e^{i\xi} \in A_m^c$ a.e. and, on the other hand, by property (4) above, we know
\[
\chi_{A_m}(e^{i\xi})k(\xi) \leq -m,
\]
for all $e^{i\xi} \in A_m$ a.e. Therefore
\begin{equation}\label{e 2}
\begin{cases}
|\widetilde{fg}_m(e^{i\xi})| = 1  & \mbox{if } e^{i\xi} \in  A_m^c
\\
|\widetilde{fg}_m(e^{i\xi})| \leq e^{-m} & \mbox{if } e^{i\xi} \in  A_m.
\end{cases}
\end{equation}
We want to show that $fg_m \raro 1$ uniformly on $A_{m_0}^c$ for each $m_0 \geq 1$. To see this, fix $m_0 \geq 1$. We know that $A_m \subsetneqq A_{m_0}$ for all $m > m_0$. Also, by property (2), there exists $M > 0$ such that
\[
|e^{i \theta}-e^{i\xi}|>M,
\]
for all $e^{i \theta} \in A_{m}$ and $e^{i\xi}\in A_{m_0}^c $. Therefore, there exists $\tilde{M} > 0$ such that
\[
|S(r e^{i\xi},\theta)\chi_{A_{m}}(e^{i \theta})|\leq \tilde{M},
\]
for all $e^{i\xi} \in {A_{m_0}^c}$ a.e. and $m>m_0$. Since
\[
\chi_{A_m}(e^{i \theta})|k(\theta)| \longrightarrow 0,
\]
pointwise, and
\[
\chi_{A_m}(e^{i\theta})|k(\theta)|\leq \left|k(\theta)\right|,
\]
on $\T$ a.e. by the dominated convergence theorem, we conclude
\[
\int_{-\pi}^{\pi}\chi_{A_m}(e^{i\theta})|k(\theta)| d\theta \longrightarrow 0.
\]
Consequently, for each $m > m_0$, we have
\[
\begin{split}
\Big|\int_{-\pi}^{\pi}S(re^{i\xi}, \theta)\chi_{A_m}(e^{i \theta})k(\theta)d\theta \Big| & \leq \tilde{M} \int_{-\pi}^{\pi} \chi_{A_m}(e^{i \theta})|k(\theta)|d\theta \longrightarrow 0,
\end{split}
\]
uniformly for all $e^{i\xi} \in A_{m_0}^c$ a.e. Since
\[
\widetilde{(fg_m)} (e^{i\xi})= \lim_{r\rightarrow 1}\operatorname{exp}\left(\frac{1}{2\pi}\int_{-\pi}^{\pi}S(re^{i\xi},
\theta)\chi_{A_m}(e^{i \theta})k(\theta)d\theta\right),
\]
we conclude that
\[
\widetilde{fg_m} \longrightarrow 1,
\]
uniformly on $A_{m_0}^c$ a.e. and hence, the claim follows. Finally, we turn to compute $\|f^2g_m -f\|$. Observe that
\[
\|f^2g_m -f\|  = \operatorname{max}\Big\{\esssup_{A_{m_0}^C}|\widetilde{f^2g}_m -\tilde{f}|, \, \esssup_{A_m} |\widetilde{f^2g}_m -\tilde{f}|\Big\}.
\]
By the fact of uniform convergence above, for sufficiently large $m$, it follows that
\[
\esssup_{A_{m_0}^C}|\widetilde{f^2g}_m -\tilde{f}| \leq \varepsilon.
\]
Therefore, for sufficiently large $m$, property (4) and \eqref{e 2} imply
\[
\esssup_{A_m} |\widetilde{f^2g}_m -\tilde{f}|\leq  e^{-m}(e^{-m}+1),
\]
and hence $\{fg_m\}_{m\geq 1}$ is a bounded approximate unit in $I_{H^\infty(\D)}(f)$.
\end{proof}

The above result provides a converse to Corollary \ref{c2}, but only within the class of functions in $Z^\infty(\D)$. However, for functions in $H^\infty(\D)$, we still can have a full converse to Corollary \ref{c2}, provided we assume Jensen's equality \cite[Chapter 10]{KH}: Recall that for each $\vp \in \clm(H^\infty(\D))$, there is a probability measure $\mu_\vp$ on $\clm(L^\infty(\T))$ such that
\[
\vp(f) = \int_{\clm(L^\infty(\T))} f \,d\mu_\vp,
\]
for all $f \in H^\infty(\D)$. It follows that
\[
\log|\vp(f)| \leq \int_{\clm(L^\infty(\T))} \log |f| \,d\mu_\vp,
\]
for all $f \in H^\infty(\D)$. This inequality is known as \textit{Jensen's inequality}. We say that $f \in H^\infty(\D)$ satisfies \textit{Jensen's equality} if
\[
\log|\vp(f)| = \int_{\clm(L^\infty(\T))} \log |f| \,d\mu_\vp.
\]
We also need to set up some notations \cite{Iz}: For an ideal $I$ in $H^\infty(\D)$, define
\[
Z(I) = \{\vp \in \clm(H^\infty(\D)): \vp(f)=0 \forall f\in I\}.
\]
Also, for each subset $K$ of $\clm(H^\infty(\D))$, we set
\[
I(K) =\{f\in H^\infty(\D):f|_K=0\}.
\]
The following theorem \cite[Theorem 3.2]{Iz} holds the key to the converse to Corollary \ref{c2}:

\begin{theorem}\label{Izz}
Let $f \in H^\infty(\D)$ be an outer function, and assume that $f$ is not invertible in $H^\infty(\D)$. Let $I = fH^\infty(\D)$ denote the ideal generated by $f$. Then
\[
\overline{I} = I(Z(I)),
\]
if and only if $f$ satisfies Jensen’s equality for every point $\vp \in \clm(H^\infty(\D))$ with $\vp(f)\neq 0$.
\end{theorem}

We are now ready for a converse to Corollary \ref{c2}:

\begin{theorem}\label{thm: conv Jensen}
Let $f\in H^\infty$. Then $I_{H^\infty(\D)}(f)$ is an $M$-ideal in $H^\infty(\D)$ if and only if $f$ is an outer function and satisfies Jensen's equality.
\end{theorem}
\begin{proof}
Suppose $I_{H^\infty(\D)}(f)$ is an $M$-ideal in $H^\infty(\D)$. By Theorem \ref{converse}, we know that $f$ is an outer function. Moreover, by  part (ii) of Theorem \ref{hd}, there exists peak set $K$ such that
\[
I_{H^\infty(\D)}(f)=I(K).
\]
We have
\[
K\subseteq Z(I_{H^\infty(\D)}(f)),
\]
and hence
\[
I_{H^\infty(\D)}(f)=I(K)\supseteq I(Z(I_{H^\infty(\D)}(f))).
\]
As $I_{H^\infty(\D)}(f)\subset I(Z(I_{H^\infty(\D)}(f)))$, it follows that
\[
I_{H^\infty(\D)}(f)=I(Z(I_{H^\infty(\D)}(f))).
\]
Theorem \ref{Izz} now implies that $f$ satisfies the Jensen's Equality.

\noindent For the reverse direction, assume that $f$ is an outer function satisfies Jensen's equality. Again, by Theorem \ref{Izz}, we conclude that $I_{H^\infty(\D)}(f)=I(Z(I_{H^\infty(\D)}(f)))$. Therefore,
\[
f^{1/n}\in {I_{H^\infty(\D)}(f)}.
\]
By using the same arguments as in Theorem \ref{prop1}, it can be shown that $f^{1/n}$ forms a bounded approximate identity in  $I_{H^\infty(\D)}(f)$. Consequently, $I_{H^\infty(\D)}(f)$ is an $M$-ideal  in $H^\infty(\D)$.
\end{proof}

Combining the above theorem with Theorem \ref{main1}, we get:

\begin{corollary}
If $f\in Z^\infty(\T)$, then $f$ satisfies Jensen's equality.
\end{corollary}

We now proceed to compute the zero set of a particular function. Recall the representations of outer functions from \eqref{eqn outer fn}: Define an integrable function $k$ on $(-\pi, \pi]$ as follows:
\[
k(\theta) =\begin{cases}
-n &\mbox{ if } \theta \in (\frac{1}{(n+1)}, \frac{1}{n}]\cup[\frac{-1}{n},\frac{-1}{(n+1)})
\\
\frac{1}{n} & \mbox{ if } \theta \in [\pi-\frac{1}{2^n}, \pi-\frac{1}{2^n}-\frac{1}{8^n}]
\\
1 & \mbox{ otherwise }.
\end{cases}
\]
Define the outer function corresponding to $k$ as
\begin{equation}\label{eqn def of f for zero}
f(z)=exp\left(\frac{1}{2\pi}\int_{-\pi}^\pi S(z,\theta)k(\theta)d\theta \right) \qquad (z \in \D),
\end{equation}
where $S(z,\theta) = \frac{e^{i\theta}+z}{e^{i\theta}-z}$. We claim that
\[
Z_{\T}(f)=\{1\}.
\]
Indeed, since
\[
|\tilde{f}(e^{i\theta})| = e^{k(\theta)},
\]
for $\theta \in (-\pi, \pi]$ a.e. it follows that
\[
|\tilde f(e^{i\theta})| \geq 1,
\]
for all $\theta \notin (-1/2, 1/2))$. Clearly, for any $z = e^{i\theta}$ such that $\theta \notin (-1/2, 1/2))$, and each neighbourhood $N_z$ of $z$, we have
\[
|\tilde{f}|^{-1}(0, 1/2)\cap N_z =\emptyset.
\]
Therefore
\[
e^{i\theta} \notin Z_{\T}(f),
\]
for all $\theta \notin (-1/2, 1/2))$. Now we consider the case
\[
\theta \in (-1/2, 1/2)) \setminus \{0\}.
\]
There exists an $n \geq 1$ such that
\[
\theta \in \Big(\frac{1}{(n+1)}, \frac{1}{n}\Big] \bigcup \Big[\frac{-1}{n},\frac{-1}{(n+1)}\Big).
\]
Hence
\[
|\tilde f(e^{i\zeta})| = e^{-n},
\]
for all $\zeta \in (\frac{1}{(n+1)}, \frac{1}{n}]\cup[\frac{-1}{n},\frac{-1}{(n+1)})$. For $z=e^{i\theta}$, and $N_z = (\frac{1}{(n+1)}, \frac{1}{n}]\cup[\frac{-1}{n},\frac{-1}{(n+1)})$, and $\varepsilon < e^{-n}$, we get
\[
|\tilde{f}|^{-1}(0, \varepsilon)\cap N_z =\emptyset.
\]
Therefore, we again conclude that $e^{i\theta} \notin Z_{\T}(f)$ for all $\theta \in (-1/2, 1/2)) \setminus \{0\}$, and hence
\[
e^{i\theta} \notin Z_{\T}(f) \qquad (\theta \in [-\pi, \pi) \setminus \{0\}).
\]
Finally, for each $\varepsilon>0$ and neighbourhood $N_1$ of $1$, we can find $n \geq 1$ such that $e^{-n}<\varepsilon$ and $(\frac{1}{(n+1)}, \frac{1}{n}]\cup[\frac{-1}{n},\frac{-1}{(n+1)})\subseteq N_1$. Therefore,
\[
\Big(\frac{1}{(n+1)}, \frac{1}{n} \Big] \bigcup \Big[\frac{-1}{n},\frac{-1}{(n+1)}\Big) \subseteq |\tilde{f}|^{-1}(0, \varepsilon)\cap N_1,
\]
and hence
\[
\mu (|\tilde{f}|^{-1}(0, \varepsilon)\cap N_1)>0.
\]
This completes the proof of the claim that $Z_{\T}(f)=\{1\}$.

The following example proves that the conclusion of Corollary \ref{inn} does not hold for arbitrary extreme points.

\begin{example}\label{extreme}
Consider the integrable function $k$ defined as
\[
k(\theta) =\begin{cases}
-n &\mbox{ if } \theta \in (\frac{1}{(n+1)}, \frac{1}{n}]\cup[\frac{-1}{n},\frac{-1}{(n+1)})
\\
\frac{1}{n} & \mbox{ if } \theta \in [\pi-\frac{1}{2^n}, \pi-\frac{1}{2^n}-\frac{1}{8^n}]
\\
0 & \mbox{ otherwise}.
\end{cases}
\]
Let $f$ denote the corresponding outer function, that is,
\[
f(z)=exp\left(\frac{1}{2\pi}\int_{-\pi}^\pi S(z,\theta)k(\theta)d\theta \right),
\]
for all $z \in \D$. We know that
\[
|f|=1,
\]
on a set $E\subset \mathbb{T}$ of positive measure. Then $f$ is an outer function which is an extreme point. Therefore, by Theorem \ref{main1}, $I_{H^\infty(\D)}(f)$ is an $M$-ideal.
\end{example}

\section{Examples}\label{sec; examples}

The objective of this section is to present some direct applications of Theorem \ref{main1} as well as some nontrivial examples of elements in $Z^\infty(\D)$. Moreover, we present exotic examples of $M$-ideals that are principal ideals generated by functions from $H^\infty(\D) \setminus Z^\infty(\D)$. We begin by highlighting, in the context of Theorem \ref{main1}, that the $M$-ideals are more explicit. For instance, given $f \in Z^\infty(\D)$, define
\begin{equation}\label{eqn I(f)}
\cli(f) = \{g \in H^\infty(\D):  Z_{\T}(f) \subseteq Z_{\T}(g) \text{ and } g \text{ has a continuous extension to }Z_{\T}(f)\}.
\end{equation}
Observe that
\[
I_{H^\infty(\D)}(f) \subseteq \cli(f).
\]
Now we assume that $f$ is outer. Pick $g \in H^\infty(\D)$ such that $\tilde{g}|_{Z_{\T}(f)}$ is continuous, and
\[
Z_{\T}(f) \subseteq Z_{\T}(g).
\]
For each open set $U$ containing $Z_{\T}(f)$, there exists a neighborhood $A \subseteq U$ of $Z_{\T}(f)$ such that the bounded approximate unit $\{fg_m\}_{m\geq 1}$ constructed in the proof of Theorem \ref{main1} converges uniformly to $1$ on $A^c$ and
\[
|\tilde{f}|, |\tilde{g}| < \varepsilon,
\]
a.e. on $A$. Since $\widetilde{fg_m}\longrightarrow 1$, uniformly on $A^c$, it follows that
\[
\esssup_{e^{i\theta} \in A^c} |\tilde{g}(e^{i \theta}) (\widetilde{fg_m})(e^{i \theta}) - \tilde{g}(e^{i \theta})| < \varepsilon,
\]
and
\[
\begin{split}
\esssup_{e^{i\theta} \in A} |\tilde{g}(e^{i \theta}) (\widetilde{fg_m})(e^{i \theta}) - \tilde{g}(e^{i \theta})| & < |\tilde{g}(\theta)|\|fg_m+1\| <2\varepsilon,
\end{split}
\]
for sufficiently large $m$. Consequently, $g\in I_{H^\infty(\D)}(f)$. This proves the following result:

\begin{corollary}\label{vanishing set}
If $f\in Z^\infty(\D)$ is an outer function, then
\[
I_{H^\infty(\D)}(f)= \cli(f).
\]
\end{corollary}

In view of $\mathbb{C}[z] \subseteq Z^\infty(\D)$, the following is immediate from Theorem \ref{main1}:

\begin{corollary}\label{polyn}
Let $p \in \mathbb{C}[z]$ be a polynomial. Then $I_{H^\infty(\D)}(p)$ is an $M$-ideal in $H^\infty(\D)$ if and only if zeros of $p$ lies on $\mathbb{C} \setminus \D$.
\end{corollary}

We now go on to the existence of some nontrivial outer functions in $Z^\infty(\D)$. We provide two classes of examples of different flavors. Recall that the disc algebra $A(\D)$ is the uniform algebra of all continuous functions on the closed disc $\overline{\D}$ that are analytic on the open disc $\D$. In other words (see \eqref{eq disc alg})
\[
A(\D) = H^\infty(\D) \cap C(\overline{\D}).
\]
It is known, as well as evident, that $A(\D) \varsubsetneqq H^\infty(\D)$. In the following, we present an example of an outer function in $Z^\infty(\D)$ that also has points of discontinuity on $\mathbb{T}$. In particular, we prove that $A(\D) \varsubsetneqq Z^\infty(\D)$.

\begin{example}
Consider the singular inner function
\[
s(z)=\exp \left(\frac{z+i}{z-i}\right) \qquad (z \in \D),
\]
and the outer function
\[
h(z)=1-z \qquad (z \in \D).
\]
Since $s$ is an inner function, $1-s$ is an outer function. Therefore, being the product of two outer functions, it follows that
\[
f(z) = (1-s(z))h(z) \qquad (z \in \D),
\]
is an outer function. Observe that
\[
f(z) = (1 - z) \Big(1- \exp \left(\frac{z+i}{z-i}\right) \Big) \qquad (z \in \D),
\]
and hence
\[
Z_{\T}(f) = \{1, -i\}.
\]
Since the support of $\tilde{s}$ is $\{i\}$, it follows that $s$ is analytically extendable to $\mathbb{T}$\textbackslash $\{i\}$. In particular, $1-s$ is continuously extendable to $\{1, -i\}$. Since $s$ is continuous on  $\mathbb{T}$\textbackslash $\{i\}$, it follows that $f$ is continuous on $Z_{\T}(f)$. Therefore $f \in Z^\infty(\D)$. However, since $s$ is discontinuous at $i$ and $h$ does not vanish at the point $i$, we conclude that $f$ is discontinuous at $i$, and subsequently, $f \notin A(\D)$.
\end{example}

Now we turn to the second example.

\begin{example}
We consider the outer function $f$ as constructed in \eqref{eqn def of f for zero}:
\[
f(z)=exp\left(\frac{1}{2\pi}\int_{-\pi}^\pi S(z,\theta)k(\theta)d\theta \right) \qquad (z \in \D),
\]
where
\[
k(\theta) =\begin{cases}
-n &\mbox{ if } \theta \in (\frac{1}{(n+1)}, \frac{1}{n}]\cup[\frac{-1}{n},\frac{-1}{(n+1)})
\\
\frac{1}{n} & \mbox{ if } \theta \in [\pi-\frac{1}{2^n}, \pi-\frac{1}{2^n}-\frac{1}{8^n}]
\\
1 & \mbox{ otherwise }.
\end{cases}
\]
We already know that $Z_{\T}(f)=\{1\}$. Observe, for sufficiently large $n$, we have
\[
\theta\in \left[\frac{-1}{n},\frac{-1}{(n+1)}\right) \cup \left(\frac{1}{(n+1)}, \frac{1}{n}\right],
\]
whenever $\theta$ gets close to $0$. Hence
\[
|\tilde{f}(e^{i\theta})| = e^{k(\theta)} = e^{-n}.
\]
As a result, $\tilde{f}$ is continuous at $1$. Now there exist sequences $\{\theta_m\}$ and $\{\gamma_m\}$ such that
\[
\theta_m \subseteq \left[\pi-\frac{1}{2^m}, \pi-\frac{1}{2^m}-\frac{1}{8^m}\right] \qquad (m \geq 1),
\]
and
\[
k(\gamma_m) =1 \qquad (m \geq 1),
\]
and $\theta_m\rightarrow \pi$, and $\gamma_m\rightarrow \pi$. This implies that $|\tilde{f}|$ is not continuous at $\pi$.
\end{example}

Therefore, there is no dearth of examples of functions in $Z^\infty(\D)$. Moreover, as we have shown above
\[
A(\D) \subsetneqq Z^\infty(\D).
\]
We now present an additional collection of natural yet exotic $M$-ideals in $H^\infty(\D)$.

\begin{theorem}
Let $f\in H^\infty(\D)$ such that $\|f\|=1$. Suppose $|\alpha|=1$ and
\[
\alpha\in \operatorname{ess-ran}f.
\]
Then $I_{H^{\infty}(\D)}(\alpha-f)$ is an $M$-ideal in $H^\infty(\D)$.
\end{theorem}
\begin{proof}
Without any loss of generality, we may assume that
\[
\alpha = 1 \in \text{ess-range}f.
\]
Fix an $\varepsilon > 0$, and let
\[
A = \{\zeta \in \T : \Tilde{f}(\zeta) \in B_\varepsilon(1)\},
\]
where $B_\varepsilon(1)$ denotes the open ball of radius $\varepsilon$ centepurple at $1$. Note that
\[
\tilde{f}(A^c) \subseteq \overline{\D} \setminus B_\varepsilon(1).
\]
Now we define a function $g \in H^\infty(\D)$ by
\[
{g} = \frac{1+{f}}{2}.
\]
For each $z \in A$, we have
\[
\begin{split}
|1- {g}(z)| & = \left|1-\frac{1+{f}(z)}{2}\right| = \frac{1}{2}\left|1-{f}(z)\right| < \frac{\varepsilon}{2},
\end{split}
\]
and hence
\[
\tilde{g}(A) \subseteq B_\frac{\varepsilon}{2}(1).
\]
Define $h \in A(\D)$ by
\[
h:=\frac{z+1}{2}.
\]
We have by the triangle inequality that $\|h\| \leq 1$. Moreover, on the boundary $\T$, we have
\[
\begin{split}
|h(e^{i\theta})|^2 & =\frac{2(1+cos(\theta))}{4} =\frac{1+cos(\theta)}{2}.
\end{split}
\]
Therefore $|h(e^{i\theta})|=1$ only when $e^{i\theta} = 1$. By continuity, there exists $c \in (0,1)$ such that
\[
\sup_{z \in \overline{\D} \setminus B_\varepsilon(1)} |h(z)| \leq c.
\]
Since
\[
\tilde{g}= h\circ \Tilde{f},
\]
it follows that
\[
|\tilde{g}| \leq c < 1,
\]
on $A^c$. Observe that
\[
I_{H^\infty(\D)}(1-f) = I_{H^\infty(\D)}(1-g).
\]
Therefore, it is enough to prove that $I_{H^\infty(\D)}(1-g)$ is an $M$-ideal in $H^\infty(\D)$.
To this end, we will construct one approximate identity in $I_{H^\infty(\D)}(1-g)$. For each $n \geq 1$, set
\[
f_n = 1-g^{n}.
\]
Since
\[
1-g^n = (1-g) (1 + g + \cdots + g^{n-1}),
\]
it follows that $f_n \in I_{H^\infty(\D)}(1-g)$. Moreover
\[
(1-g)f_n-(1-g) = -(1-g)g^n.
\]
By using the property of radial limits, we conclude
\[
|(1-g)g^{n}| \leq
\begin{cases}
\frac{\varepsilon}{2} & \text{on } A
\\
2c^{n} & \text{on } A^c.
\end{cases}
\]
By the maximum modules principle, we find that $\{f_n\}_{n \geq 1}$ is a bounded approximate identity in $I_{H^\infty(\D)}(1-g)$, and hence $I_{H^\infty(\D)}(1-g)$ is an $M$-ideal in $H^\infty(\D)$.
\end{proof}

Now we turn to the final example in this section. The class of functions given in the above theorem has the potential for the existence of singly generated $M$-ideal in $H^\infty(\D)$ that may not belong to $Z^\infty(\D)$. Indeed, the following example shows an $M$-ideal in $H^\infty(\D)$ generated by an outer function not belonging to the class $Z^\infty(\D)$.   
 
 
 

\begin{example}\label{examp: ideal not in Z}
Define an integrable function $k$ on $(- \pi, \pi]$ by
\[
k (\theta) =
\begin{cases}
-(\theta^2+1) & \text{if } -\pi <\theta < 0
\\
- \theta & \text{if }  0\leq \theta \leq \pi,
\end{cases}
\]
and consider the corresponding outer function
\[
f(z)=exp\left(\frac{1}{2\pi}\int_{-\pi}^\pi S(z,\theta) k(\theta)d\theta \right)\qquad (z \in \D).
\]
Now, observe that
\[
\begin{split}
1 \in \text{ess-range}|\tilde{f}|,
\end{split}
\]
and hence there exists $\alpha \in \T$ such that
\[
1\in Z_\T(\alpha - f).
\]
On the other hand, if $\theta \raro 0$ from the left, then
\[
k(\theta) \longrightarrow  -1,
\]
and hence
\[
|\tilde{f}|  \longrightarrow k(\theta) = \frac{1}{e}.
\]
Therefore, $\alpha- f$ is not continuous at $1$ (or on $Z_\T(\alpha - f)$) but, in view of the above theorem, $I_{H^\infty(\D)}(\alpha - f)$ remains an $M$-ideal in $H^\infty(\D)$. 
\end{example}

The examples and results shown in this section indicate that the structure of $M$-ideals in $H^\infty(\D)$ is intricate, even when considering $M$-ideals that are principal generated. 
 

\section{Finitely generated $M$-ideals}\label{sec; finitely gen}

In the preceding sections, we examined representations of principal ideals generated by functions from $Z^\infty(\D)$. Specifically, we showed in Theorem \ref{main1} that the ideal $I_{H^\infty(\D)}(f)$ is an $M$-ideal for some function $f \in Z^\infty(\D)$ if and only if $f$ is an outer function. In this section, we consider $M$-ideals in $H^\infty(\D)$ that are finitely generated by functions from $Z^\infty(\D)$. Given $m$ functions $\{f_1, \ldots, f_m\} \subseteq H^\infty(\D)$, denote by $I_{H^\infty(\D)}(f_1, \ldots, f_m)$ the ideal generated by $\{f_1, \ldots, f_m\}$.

\begin{theorem}\label{fg}
Let $\{f_1,\ldots, f_m\} \subseteq Z^\infty(\D)$. If $I_{H^\infty(\D)}(f_1, \ldots, f_m)$ is an $M$-ideal in $H^\infty(\D)$, then there exists an outer function $f \in Z^\infty(\D)$ such that
\[
I_{H^\infty(\D)}(f_1, \ldots, f_m) = I_{H^\infty(\D)}(f).
\]
\end{theorem}
\begin{proof}
Suppose $I_{H^\infty(\D)}(f_1, \ldots, f_m)$ is an $M$-ideal in $H^\infty(\D)$. Let
\[
f_j = f_{Ij} f_{Oj},
\]
is the inner-outer factorization, where $f_{I,j}$ is the inner factor and $f_{O,j}$ is the outer factor of $f_j$ for all $j=1, \ldots, m$. In view of
\[
f_{Ij} f_{Oj} \in I_{H^\infty(\D)}(f_1, \ldots, f_m) \qquad (j=1, \ldots, m),
\]
we know, by Theorem \ref{prop1}, that
\[
f_{Oj} \in I_{H^\infty(\D)}(f_1, \ldots, f_m) \qquad (j=1, \ldots, m),
\]
and hence
\[
I_{H^\infty(\D)}(f_{O1}, \ldots, f_{Om}) \subseteq I_{H^\infty(\D)}(f_1, \ldots, f_m).
\]
By using the inner-outer factorizations of $f_i$'s, the other set inclusion becomes trivial. Then
\[
I_{H^\infty(\D)}(f_{O1}, \ldots, f_{Om}) = I_{H^\infty(\D)}(f_1, \ldots, f_m).
\]
Therefore, without any loss of generality, we may assume that $f_j$'s are outer functions in $H^\infty(\D)$. Our goal is to prove that
\[
I_{H^\infty(\D)}(f_1, \ldots, f_m) = I_{H^\infty(\D)}(f),
\]
for some outer function $f \in Z^\infty(\D)$. We prove the assertion for $m=2$. The complete proof follows similarly by induction on $m$. Therefore, our revised goal is to prove the following fact: Assuming that $I_{H^\infty(\D)}(f_1, f_2)$ is an $M$-ideal in $H^\infty(\D)$ for some $f_1, f_2 \in Z^\infty(\D)$, there exists an outer function $f \in Z^\infty(\D)$ such that
\[
I_{H^\infty(\D)}(f_1, f_2) = I_{H^\infty(\D)}(f).
\]
To this end, first, assume that
\[
Z({f_1}) \cap Z({f_2}) = \emptyset.
\]
Since $f_j \in Z^\infty(\D)$ is outer, by Theorem \ref{main1}, it follows that $I_{H^\infty(\D)}(f_j)$ is an $M$-ideal, $j=1,2$, and hence there exist bounded approximate units, say (by an abuse of notation) $\{\vp_m\}_{m\geq 1}$ and $\{\psi_m\}_{m\geq 1}$ in $I_{H^\infty(\D)}(f_1)$ and $I_{H^\infty(\D)}(f_2)$, respectively, as constructed in the proof of Theorem \ref{main1}. We highlight the following key properties: For any $\varepsilon>0$, there exists a neighborhood $U_1\subseteq \T$ of $Z_{\T}({f_1})$ such that
\begin{equation}\label{eqn: phi 1}
|{\widetilde{\vp_m}} - 1| <\varepsilon,
\end{equation}
on $U_1^c$, and
\begin{equation}\label{eqn: phi 2}
|{\widetilde{\vp_m}}|<\varepsilon,
\end{equation}
on $U_1$ for sufficiently large $m$. Similarly, there exists a neighborhood $U_2\subseteq \T$ of $Z_{\T}({f_2})$ such that $|{\widetilde{\vp_m}} - 1| <\varepsilon$ on $U_1^c$, and $|{\widetilde{\vp_m}}|<\varepsilon$ on $U_2$ for sufficiently large $m$. Since $Z_{\T}(f_1)$ and $Z_{\T}(f_2)$ are compact, without any loss of generality, we may assume that
\[
U_1\cap U_2 =\emptyset.
\]
Therefore, for sufficiently large $p, q \geq 1$, we have
\[
\Big|\frac{\widetilde{\vp}_{p}+\widetilde{\psi}_q}{2} - 1 \Big| < \varepsilon,
\]
on $U_1^c\cup U_2^c = \T$. That is
\[
1- \varepsilon < \Big|\frac{\widetilde{\vp}_{p}+\widetilde{\psi}_q}{2}\Big|,
\]
on $\T$. But
\[
\frac{\widetilde{\vp}_{p}+\widetilde{\psi}_q}{2} \in I_{H^\infty(\D)} (f_1, f_2),
\]
and hence, by Theorem \ref{prop1}, we conclude
\[
I_{H^\infty(\D)} (f_1, f_2) = I_{H^\infty(\D)} (1).
\]
Now we assume that
\[
Z_{\T}(f_1) \cap Z_{\T}(f_2) \neq \emptyset.
\]
For each $p,q \geq 1$, define
\[
\zeta_{p,q} = \vp_p + \psi_q - \vp_p \psi_q.
\]
Then
\[
\zeta_{p,q} \in I_{H^\infty(\D)}(f_1, f_2) \qquad (p,q \geq 1).
\]
We claim that $\{\zeta_{p,q}\}_{p,q \geq 1}$ is a bounded approximate unit in $I_{H^\infty(\D)}(f_1, f_2)$. It is clear that $\{\zeta_{p,q}\}_{p,q \geq 1}$ is a bounded set. It is now enough to prove that
\[
\lim_{p,q} \zeta_{p,q} f_j = f_j,
\]
for all $j=1,2$ (with respect to the lexicographic ordering). Fix $\varepsilon > 0$. Since $\{\vp_m\}_{m\geq 1}$ is a bounded approximate identity in $I_{H^\infty(\D)}(f_1)$, there exists $p_0 \geq 1$ such that
\[
\|f_1- \vp_pf_1\|_\infty \leq \varepsilon \qquad (p \geq p_0).
\]
We compute
\begin{align*}
\|\zeta_{p,q} f_1 - f_1\| & = \| \vp_p f_1 + \psi_q f_1 - \vp_p \psi_q f_1 - f_1\|
\\
& = \|\vp_p f_1 - f_1 + \psi_q f_1 - \vp_p \psi_q f_1 \|
\\
& \leq \|\vp_p f_1 - f_1\| + \|\psi_q\| \|f_1 - \vp_p f_1\|
\\
& \leq \varepsilon(1+ \|\psi_q\|)
\\
& \leq \varepsilon (1+ M),
\end{align*}
for all $p \geq p_0$, where (recall that $\{\psi_q\}_{q \geq 1}$ is a bounded sequence)
\[
M = \sup_{q} \|\psi_q\|.
\]
Assume, without any lose of generality, that $\|\vp_p\|_\infty \leq M$ for all $p \geq 1$. Then, a similar computation as above implies the existence of a number $q_0 \geq1$ such that
\[
\|\zeta_{p,q} f_2 - f_2\| \leq \varepsilon (1+ M),
\]
for all $q \geq q_0$. Hence
\[
\|\zeta_{p,q} f_j - f_j\| \leq \varepsilon (1+ M),
\]
for all $j=1,2$, and $p \geq p_0$ and $q \geq q_0$. Then $\{\zeta_{p,q}\}_{p,q \geq 1}$ is a bounded approximate identity in, and consequently $I_{H^\infty(\D)}(f_1, f_2)$ is an $M$-ideal. It remains to prove that $I_{H^\infty(\D)}(f_1, f_2) = I_{H^\infty(\D)}(f)$  for some outer function $f \in H^\infty(\D)$. First, let us define
\[
\cli = \{h\in H^\infty(\D): Z_{\T}(f_1) \cap Z_{\T}(f_2) \subseteq Z_{\T}(h) \text{ and } h|_{Z_{\T}(f_1) \cap Z_{\T}(f_2)} \text{ is continuous}\}.
\]
We claim that $I_{H^\infty(\D)} (f_1, f_2) = \cli$. Let $h\in \cli$. We want to show that
\[
\lim_{p,q} \widetilde{\zeta_{p,q}} \tilde{h} = \tilde{h}.
\]
Let $\varepsilon>0$. Since $h \in C(Z(f_1) \cap Z(f_2))$, we can find an open set $U\subset \T$ such that
\[
Z(f_1) \cap Z(f_2) \subset U,
\]
and
\[
|\tilde{h}|<\frac{\varepsilon}{4},
\]
on $U$. As in \eqref{eqn: phi 1} and \eqref{eqn: phi 2}, there exist open sets $U_1, U_2 \subseteq \T$ such that $Z(f_j) \subseteq U_j$, $j=1,2$, and
\[
U_1\cap U_2\subseteq U.
\]
Moreover, we have the following properties:
\[
\frac{\varepsilon}{2\|\tilde{h}\|} >
\begin{cases}
|\widetilde{\vp_p}-1|  & \mbox{on } U_1^c
\\
|\widetilde{\vp_p}| & \mbox{on } U_1,
\end{cases}
\]
and
\[
\frac{\varepsilon}{2\|\tilde{h}\|} >
\begin{cases}
|\widetilde{\psi_q}-1| & \mbox{on } U_2^c
\\
|\widetilde{\psi_q}| & \mbox{on } U_2,
\end{cases}
\]
for sufficiently large $p$ and $q$. Also
\[
\|\widetilde{\zeta_{p,q}} \tilde{h} - \tilde{h}\| = \|( \widetilde{\vp_p}\tilde{h}+\widetilde{\psi_q}\tilde{h}-\widetilde{\vp_p}\widetilde{\psi_q}\tilde{h})-\tilde{h}\|,
\]
where, for sufficiently large $p$ and $q$, we have
\begin{align*}
\|( \widetilde{\vp_p}\tilde{h}+\widetilde{\psi_q}\tilde{h}-\widetilde{\vp_p}\widetilde{\psi_q}\tilde{h})-\tilde{h}\|
&
\leq\begin{cases}
4\|\tilde{h}\| & \text{ on } U_1\cap U_2\\
\|\widetilde{\vp_p}\|\|\tilde{h}-\widetilde{\psi_q}\tilde{h}\|+\|\widetilde{\psi_q}\tilde{h}-\tilde{h}\| & \text{ on } U_1\cap U_2^c\\
\|\widetilde{\psi_q}\|\|\tilde{h}-\widetilde{\vp_p}\tilde{h}\|+\|\widetilde{\vp_p}\tilde{h}-\tilde{h}\|& \text{ on } U_1^c\cap U_2\\
\|\widetilde{\psi_q}\|\|\tilde{h}-\widetilde{\vp_p}\tilde{h}\|+\|\widetilde{\vp_p}\tilde{h}-\tilde{h}\|& \text{ on } U_1^c\cap U_2^c\\
\end{cases}   \\
&\leq\begin{cases}
\varepsilon & \text{ on } U_1\cap U_2\\
\frac{\varepsilon}{2\|h\|} \cdot \|\tilde{h}\|+\frac{\varepsilon}{2\|\tilde{h}\|} \cdot \|\tilde{h}\|& \text{ on } U_1\cap U_2^c\\
 \frac{\varepsilon}{2\|\tilde{h}\|} \cdot \|\tilde{h}\|+\frac{\varepsilon}{2\|\tilde{h}\|} \cdot \|\tilde{h}\|& \text{ on } U_1^c\cap U_2\\
\frac{\varepsilon}{2\|\tilde{h}\|} \cdot \|\tilde{h}\|+\frac{\varepsilon}{2\|\tilde{h}\|} \cdot \|\tilde{h}\|& \text{ on } U_1^c\cap U_2^c
\end{cases}
\\
& < \varepsilon,
\end{align*}
and hence
\[
h\in I_{H^\infty(\D)} (f_1, f_2).
\]
On the other hand, it is clear that every function in $I(f_1, f_2)$ vanishes and is continuous on $Z({f_1}) \cap Z({f_2})$. Therefore
\[
I_{H^\infty(\D)} (f_1, f_2) = \cli,
\]
which completes the proof of the claim. Finally, since $f_1$ and $f_2$ are continuous on $Z({f_1}) \cap Z({f_2})$ and since $Z({f_1}) \cap Z({f_2})$ is a closed subset of $\T$ with Lebesgue measure $0$, by a theorem of Fatou, there exists an outer function $f\in A(\D)$ such that $Z_f = Z(f_1) \cap Z(f_2)$. But Corollary \ref{vanishing set} implies
\[
\cli(f)= I(f_1, f_2),
\]
which completes the proof of the result.
\end{proof}

It is perhaps noteworthy to mention that the proof of the above theorem remains valid even when an infinite number of generators are adopted. However, in order to execute the proof in its present form, one needs to have the crucial finite intersection property of compact sets.

We wrap up this paper with a more specific question. In view of Theorem \ref{prop1}, we know that all $M$-ideals are analytic primes. We also note that the classification of closed prime ideals in $H^\infty(\D)$ is unknown (however, see \cite{Bur, Hart, Iz, RM2}). An intriguing question thus emerges concerning the classification of analytic primes in $H^\infty(\D)$ or a more general uniform algebra.





\vspace{0.2in}

\noindent\textbf{Acknowledgement:}
The research of the second named author is supported in part by ANRF, DST, Government of India (File No: ANRF/ARGM/2025/000130/MTR). The reserach of the third author was supported in part by ISI visiting scientist fellowship and JC Bose National Fellowship of Professor Debashish Goswami (ISI Kolkata).

\end{document}